\documentclass[11pt,reqno]{amsart}
\usepackage[margin=1in]{geometry}
\usepackage{algorithm,algpseudocode}
\usepackage[fleqn,tbtags]{mathtools}
\usepackage{pifont,enumerate}
\usepackage{verbatim}
\usepackage{hyperref,xcolor}
\hypersetup{
    colorlinks,
    linkcolor={red!50!black},
    citecolor={blue!50!black},
    urlcolor={blue!80!black}
}

\newcommand{\norm}[1]{\lVert#1\rVert}
\newcommand{\abs}[1]{\lvert#1\rvert}

\newcommand\restr[2]{{% we make the whole thing an ordinary symbol
  \left.\kern-\nulldelimiterspace % automatically resize the bar with \right
  #1 % the function
  \vphantom{\big|} % pretend it's a little taller at normal size
  \right|_{#2} % this is the delimiter
  }}

\newcommand{\tp}{{\scriptscriptstyle\mathsf{T}}}
\newcommand{\p}{{\scriptscriptstyle+}}

\DeclareMathOperator{\rank}{rank}
\DeclareMathOperator{\tr}{tr}
\DeclareMathOperator{\diag}{diag}

\DeclareMathOperator*{\argmin}{argmin}

\theoremstyle{definition}
\newtheorem{theorem}{Theorem}[section]

\newtheorem{lemma}[theorem]{Lemma}
\newtheorem{corollary}[theorem]{Corollary}

\numberwithin{equation}{section}

\begin{document}

\title{Generalized matrix nearness problems}
\author{Zihao Li}
\address{Operations Research and Financial Engineering, Princeton University, Princeton, NJ  08540}
\email{zl1665@princeton.edu}
\author{Lek-Heng Lim}
\address{Computational and Applied Mathematics Initiative, University of Chicago, Chicago, IL 60637}
\email{lekheng@uchicago.edu}

\begin{abstract}
We show that the global minimum solution of $\norm{A - BXC}$ can be found in closed-form with singular value decompositions and generalized singular value decompositions for a variety of constraints on $X$ involving rank, norm, symmetry, two-sided product, and prescribed eigenvalue. This extends the solution of Friedland--Torokhti for the generalized rank-constrained approximation problem to other constraints as well as provides an alternative solution for rank constraint in terms of singular value decompositions. For more complicated constraints on $X$ involving structures such as Toeplitz, Hankel, circulant, nonnegativity, stochasticity, positive semidefiniteness, prescribed eigenvector, etc, we prove that a simple iterative method is linearly and globally convergent to the global minimum solution.
\end{abstract}



\maketitle

\section{Introduction}

In \cite{Fri}, Friedland and Torokhti found a closed-form analytic solution for the generalized rank constrained matrix approximation problem
\begin{equation}\label{eq:Fri}
\min_{\rank(X) \le r}\;\norm{A-BXC}
\end{equation}
with $A \in \mathbb{R}^{m \times n}$, $B \in \mathbb{R}^{m \times p}$, $C \in \mathbb{R}^{q \times n}$, and $\norm{\,\cdot\,}$ the Frobenius norm, generalizing the celebrated result of Eckhart and Young \cite{Eck}. We extend this work in several ways, replacing the rank constraint $\rank(X) \le r$ by
\begin{dingautolist}{192}
\item\label{it:norm} norm constraint: $\lVert X \rVert \le \rho $ for a given $\rho > 0$;
\item\label{it:le} two-sided product constraint: $FXG = H$ for given matrices $F,G,H$;
\item\label{it:eig} spectral constraints: $X$ has a prescribed eigenvalue $\lambda$ or a prescribed eigenvector $v  \ne 0$;
\item\label{it:symm} symmetry constraints: $X$ is symmetric or skew-symmetric;
\item\label{it:struct} structure constraints: $X$ is Toeplitz, Hankel, or circulant;
\item\label{it:pos} positivity constraints: $X$ is positive semidefinite, correlation, nonnegative, stochastic, or doubly stochastic.
\end{dingautolist}
Note that \ref{it:le} includes the important special case $Xg = h$ for given vectors $g,h$.
We shall provide closed-form analytic solutions for \ref{it:norm}--\ref{it:symm}, using \textsc{svd} for \ref{it:norm}--\ref{it:eig}  and \textsc{gsvd} for \ref{it:symm}, with the exception of the prescribed eigenvector problem --- for this and for \ref{it:struct} and \ref{it:pos}, we prove that an iterative algorithm, when applied to these problems, is
\begin{enumerate}[\upshape (i)]
\item globally convergent, i.e., converges for any initial point;
\item linearly convergent, i.e., error decreases exponentially to zero;
\item provably convergent, i.e., converges to a global minimizer, not just a stationary point or local minimizer.
\end{enumerate}
As an addendum, we provide a simpler alternative solution to \eqref{eq:Fri} in terms of singular value decompositions. While it is analytically equivalent to the solution of \cite{Fri} in terms of projection matrices and pseudoinverses, which are numerically unstable to compute, an obvious advantage of our solution is that it is stably computable via singular value decompositions.

We emphasize that we do not treat the problems \ref{it:norm}--\ref{it:pos} as constrained optimization problems and then apply general purpose nonlinear or convex optimization methods. While these problems are stated as optimization problems, our solutions are firmly rooted in numerical linear algebra, in the tradition of \cite{Gol1,Gol2,Hig4,Hig2,Hig3,Hig1,Kell,Rao,Wil2}, and crucially rely on the matrix structures in these problems. In particular, none of our methods would involve taking derivatives; all of them are zeroth order method from the perspective of optimization.

Throughout this article we assume that the dimensions of the matrices $A\in\mathbb{R}^{m \times n}$, $B\in\mathbb{R}^{m \times p}$, $C\in\mathbb{R}^{q \times n}$ satisfy $m\ge p$ and  $n\ge q$ since otherwise we may simply add rows of zeros to $A$ and $B$ or columns of zeros to $A$ and $C$. While all results are stated over $\mathbb{R}$, it is routine to extend them to $\mathbb{C}$.

\section{Closed-form solutions via SVD}\label{sec:frame}

We simplify the objective function $\norm{A-BXC}$ via singular value decomposition and the orthogonal invariance of Frobenius norm and matrix rank. Let 
\[
B=U_B\begin{bmatrix}\Sigma_B\\ 0\end{bmatrix}V_B^\tp, \qquad C=U_C\begin{bmatrix}\Sigma_C & 0\end{bmatrix}V_C^\tp
\]
be singular value decompositions. Since the Frobenius norm is invariant under left and right multiplications by orthogonal matrices, we have
\[
\norm{A-BXC} = \biggl\lVert U_B^\tp AV_C-\begin{bmatrix}\Sigma_B\\ 0\end{bmatrix} V_B^\tp XU_C\begin{bmatrix}\Sigma_C& 0\end{bmatrix} \biggr\rVert
= \biggl\lVert\widetilde{A}-\begin{bmatrix}\Sigma_B\\ 0\end{bmatrix}\widetilde{X}\begin{bmatrix}\Sigma_C& 0\end{bmatrix} \biggr\rVert
\]
where $\widetilde{A} \coloneqq U_B^\tp AV_C$ and $\widetilde{X} \coloneqq V_B^\tp XU_C$.
Let $\rank(B)=s$, $\rank(C)=t$, $S_B \coloneqq \diag\bigl(\sigma_1(B),\dots,\sigma_s(B)\bigr)$, $S_C \coloneqq \diag\bigl(\sigma_1(C),\dots,\sigma_t(C)\bigr)$. Then
\[
\begin{bmatrix}\Sigma_B\\ 0\end{bmatrix}=\begin{bmatrix}
S_B & 0\\
0 & 0
\end{bmatrix},\quad
\begin{bmatrix}\Sigma_C& 0\end{bmatrix}=\begin{bmatrix}
S_C & 0\\
0 & 0
\end{bmatrix}.
\]
Partition $\widetilde{X}$ and $\widetilde{A}$ as
\[
\widetilde{X}=\begin{bmatrix}
X_{11} & X_{12} \\
X_{21} & X_{22}
\end{bmatrix},\quad
\widetilde{A}=\begin{bmatrix}
A_{11} & A_{12} \\
A_{21} & A_{22}
\end{bmatrix}
\]
with $X_{11},A_{11}\in\mathbb{R}^{s \times t}$, and we obtain
\begin{align}
\biggl\lVert \widetilde{A}-\begin{bmatrix}\Sigma_B\\ 0\end{bmatrix}\widetilde{X}\begin{bmatrix}\Sigma_C& 0\end{bmatrix} \biggr\rVert^2
&= \biggl\lVert \begin{bmatrix}A_{11} & A_{12} \\ A_{21} & A_{22} \end{bmatrix}-\begin{bmatrix}S_B & 0 \\ 0 & 0\end{bmatrix}\begin{bmatrix}X_{11} & X_{12} \\ X_{21} & X_{22}\end{bmatrix}\begin{bmatrix}S_C & 0 \\ 0 & 0\end{bmatrix}\biggr\rVert^2 \nonumber \\
&=\biggl\lVert\begin{bmatrix}A_{11}-S_BX_{11}S_C & A_{12} \\ A_{21} & A_{22}\end{bmatrix}\biggr\rVert^2 \nonumber \\
&=\norm{A_{11}-S_BX_{11}S_C}^2+\norm{A_{12}}^2+\norm{A_{21}}^2+\norm{A_{22}}^2. \label{eq:frame}
\end{align}
Note in particular that $X_{12}$, $X_{21}$, $X_{22}$ do not appear in the final expression and we are free to choose them within whatever constraint we impose on $X$.

\subsection{Generalized rank constrained approximation}\label{sec:rank}

As an illustration, we consider the Friedland--Torokhti rank approximation problem \cite{Fri}
\[
\min_{\rank(X)\le r}{\norm{A-BXC}}.
\]
We may set $X_{12}$, $X_{21}$, $X_{22}$ to be zero matrices in \eqref{eq:frame} and since
\[
\rank(S_BX_{11}S_C)=\rank(X_{11})=\rank(\widetilde{X})=\rank(X),
\]
we only need to solve 
\[
\min_{\rank(S_BX_{11}S_C)\le r}\norm{A_{11}-S_BX_{11}S_C}.
\]
It immediately follows from Eckart--Young Theorem that the solution is $X_{11}=S_B^{-1}U_r \Sigma_r V_r^\tp S_C^{-1}$, with singular value decomposition $A_{11} = U \Sigma V^\tp $ and $U_r \Sigma_r V_r^\tp$ the best rank-$r$ approximation of $A_{11}$. In fact, given that we have set all free parameters in $\widetilde{X}$ to zero, this actually gives the minimum-norm solution. We summarize our solution in Algorithm~\ref{algo:gr}.

\begin{algorithm}[htb]
  \caption{Generalized rank constrained approximation}\label{algo:gr}
  \begin{algorithmic}[1]
    \Require $A\in\mathbb{R}^{m \times n}$, $B\in\mathbb{R}^{m \times p}$, $C\in\mathbb{R}^{q \times n}$, $r\ge 0$;
    \State compute singular value decompositions $B=U_B\begin{bsmallmatrix}\Sigma_B\\ 0\end{bsmallmatrix}V_B^\tp $ and $C=U_C\begin{bmatrix}\Sigma_C& 0\end{bmatrix}V_C^\tp $;
    \State compute $A_{11}$ from
\[
U_B^\tp AV_C = \begin{bmatrix}A_{11} & A_{12} \\ A_{21} & A_{22} \end{bmatrix};
\]
    \State compute singular value decomposition $A_{11} = U \Sigma V^\tp$;
    \State set $S_B=\diag\bigl(\sigma_1(B),\dots,\sigma_s(B)\bigr)$ and $S_C=\diag\bigl(\sigma_1(C),\dots,\sigma_t(C)\bigr)$;
    \Ensure
\[
X=V_B\begin{bmatrix}S_B^{-1} U_r \Sigma_r V_r^\tp S_C^{-1} & 0\\
0 & 0\end{bmatrix}U_C^\tp.
\]
  \end{algorithmic}
\end{algorithm}

\subsection{Generalized prescribed eigenvalue approximation}

A consequence of our previous solution is the solution to the prescribed eigenvalue approximation problem in \ref{it:eig}. This problem requires square matrices, so $A \in \mathbb{R}^{n \times n}$, $B \in \mathbb{R}^{n \times p}$, $C \in \mathbb{R}^{p \times n}$. Let $\lambda(X)$ denote the spectrum of $X \in \mathbb{R}^{p \times p}$ and let $\lambda$ be a given value. We want
\[
\min_{\lambda\in\lambda(X)}{\norm{A-BXC}}.
\]
The special case where $B = C = I$ was famously discussed by Wilkinson  in \cite{Wil}.

Since 
\[
\lambda\in\lambda(X) \quad \Leftrightarrow \quad \rank(X-\lambda I)\le p-1,
\]
we have
\begin{align*}
\min_{\lambda\in\lambda(X)}{\norm{A-BXC}}&=\min_{\lambda\in\lambda(X)}{\norm{A-\lambda BC-B(X-\lambda I)C}}\\
&=\min_{\rank(Y)\le p-1}{\norm{\widetilde{A}-BYC}}
\end{align*}
where $\widetilde{A}=A-\lambda BC$, and the problem reduces to the one in Section~\ref{sec:rank}.

\subsection{Generalized norm constrained approximation}
\label{sec:norm}
The problem in \ref{it:norm}:
\[
\min_{\left\lVert X\right\rVert\le \rho}{\norm{A-BXC}}
\]
is of course a special case of a norm constrained least squares problem if we ignore the fact that the variables $x_{ij}$'s come from a matrix. It may thus be solved using general techniques in \cite[Section~5.3]{Bjo}. The advantage of our approach is that it preserves the matrix structure of the problem so that, for example, we just need to decompose matrices $B$ and $C$ instead of the matrix $C \otimes B$, which is an order of magnitude larger.

By \eqref{eq:frame}, we may set $X_{12}$, $X_{21}$, $X_{22}$ to be zero matrices. Then
\[
\norm{X_{11}}=\norm{X}\le \rho.
\]
If $\norm{S_B^{-1}A_{11}S_C^{-1}}\le \rho$, then the solution is simply $X_{11} =S_B^{-1}A_{11}S_C^{-1}$. So we may suppose that $\norm{S_B^{-1}A_{11}S_C^{-1}}>\rho$ and in which case, the solution must lie on the boundary, i.e., $\norm{X_{11}}=\rho$. To see this, note that if $\norm{X_{11}}<\rho$, then $X_{11} \ne S_B^{-1}A_{11}S_C^{-1}$ since $ \norm{S_B^{-1}A_{11}S_C^{-1}} > \rho$. Hence there exists some $t\in(0,1)$ so that $\norm{t S_B^{-1}A_{11}S_C^{-1} + (1-t)X_{11}}=\rho$ and
\[
\norm{A_{11}-S_B[tS_B^{-1}A_{11}S_C^{-1}+(1-t){X_{11}}]S_C} =(1-t)\norm{A_{11}-S_BX_{11}S_C} < \norm{A_{11}-S_BX_{11}S_C},
\]
contradicting the minimality of $X_{11}$.

So the inequality constraint may be replaced by an equality constraint $\norm{X_{11}} =\rho$ and standard theory of Lagrange multiplier \cite[Chapter~14]{Pro} applied to
\[
\norm{A_{11}-S_BX_{11}S_C}^2+\lambda(\norm{X_{11}}^2-\rho^2)
\]
gives
\begin{equation}\label{eq:5}
\left\{\begin{aligned}
S_B(S_BX_{11}S_C-A_{11})S_C+\lambda X_{11}&=0,\\
\norm{X_{11}}^2&=\rho^2.
\end{aligned}\right.
\end{equation}
If $(X_1, \lambda_1)$ and $(X_2, \lambda_2)$ are both solutions to \eqref{eq:5}, then
\begin{equation}\label{eq:6}
\begin{split}
S_B(S_BX_1S_C-A_{11})S_C+\lambda_1 X_1=0,\\
S_B(S_BX_2S_C-A_{11})S_C+\lambda_2 X_1=0,
\end{split}
\end{equation}
and $\norm{X_1}^2=\norm{X_2}^2=\rho^2$. Left multiply the first equation in \eqref{eq:6} by $X_1^\tp$ and the second by $X_2^\tp$, then take the trace and subtract, we get
\begin{equation}\label{eq:7}
\norm{S_BX_2S_C}^2-\norm{S_BX_1S_C}^2-\bigl(\tr(X_2^\tp S_BA_{11}S_C)-\tr(X_1^\tp S_BA_{11}S_C)\bigr)=\lambda_1\norm{X_1}^2-\lambda_2\norm{X_2}^2.
\end{equation}
Left multiply the first equation in \eqref{eq:6} by $X_2^\tp$ and the second by $X_1^\tp$, then take trace and subtract, we get
\begin{equation}\label{eq:8}
-\bigl(\tr(X_2^\tp S_BA_{11}S_C)-\tr(X_1^\tp S_BA_{11}S_C)\bigr) = -(\lambda_1 -  \lambda_2) \tr(X_1^\tp X_2).
\end{equation}
Adding \eqref{eq:7} and \eqref{eq:8}, and noting that $\norm{X_1}^2=\norm{X_2}^2$, we get
\[
\norm{A_{11}-S_BX_2S_C}-\norm{A_{11}-S_BX_1S_C}=\frac{\lambda_1-\lambda_2}{2}\norm{X_1-X_2}^2.
\]
Hence
\[
\lambda_1 > \lambda_2 \quad \Rightarrow\quad \norm{A_{11}-S_BX_2S_C}>\norm{A_{11}-S_BX_1S_C}.
\]
The smallest $\norm{A_{11}-S_BX_{11}S_C}$ corresponds to the largest $\lambda$. Thus we seek the solution $(X_{11},\lambda)$ to \eqref{eq:5} with the largest $\lambda$.
Let $A_{11}=(a_{ij})$, $X_{11}=(x_{ij})$, $\sigma_{ij}=\sigma_i(B)\sigma_j(C)$. The solution to the first equation \eqref{eq:5} is 
\[
x_{ij}=\frac{a_{ij}}{\sigma_{ij}+\lambda\sigma_{ij}^{-1}}.
\]
Plugging into the second equation in \eqref{eq:5}, we obtain the \emph{secular equation} 
\begin{equation}\label{eq:10}
f(\lambda) \coloneqq \sum_{i=1}^s\sum_{j=1}^t \frac{a_{ij}^2}{(\sigma_{ij}+\lambda\sigma_{ij}^{-1})^2}=\rho^2.
\end{equation}

The secular equation is a ubiquitous univariate nonlinear equation in matrix computation, see \cite[Section~12.1.1]{GVL} or \cite[Section~5.3.3]{Dem} for a discussion of its properties. In particular $f$ has poles at $-\sigma_{ij}^{2}$ and $\lim_{\lambda\rightarrow\pm \infty}f(\lambda)=0$ so  $f(\lambda) = \rho^2$ has only be a finite number of solutions. We may find all real roots of $f(\lambda) = \rho^2$ using any standard univariate root finder, e.g., Newton--Raphson, regula falsi, Brent, etc, and identify the largest root. A well-known trick is to instead apply the root finder to $1/f(\lambda)=1/\rho^2$ as $1/f$ is close to linear in the vicinity of a root and convergence will be extremely fast.
We summarize this solution in Algorithm~\ref{algo:gn}.

\begin{algorithm}[htb]
  \caption{Generalized norm constrained approximation}\label{algo:gn}
  \begin{algorithmic}[1]
    \Require $A\in\mathbb{R}^{m \times n}$, $B\in\mathbb{R}^{m \times p}$, $C\in\mathbb{R}^{q \times n}$, $\rho>0$;
    \State compute singular value decompositions $B=U_B\begin{bsmallmatrix}\Sigma_B\\ 0\end{bsmallmatrix}V_B^\tp $ and $C=U_C\begin{bmatrix}\Sigma_C& 0\end{bmatrix}V_C^\tp $;
    \State  compute $A_{11}$ from
\[
U_B^\tp AV_C = \begin{bmatrix}A_{11} & A_{12} \\ A_{21} & A_{22} \end{bmatrix};
\]
    \State set $S_B=\diag\bigl(\sigma_1(B),\dots,\sigma_s(B)\bigr)$ and $S_C=\diag\bigl(\sigma_1(C),\dots,\sigma_t(C)\bigr)$;
    \State set $\sigma_{ij}=\sigma_i(B)\sigma_j(C)$;
    \State calculate largest root $\lambda$ of $f(\lambda)=\rho^2$;
    \State set $X_{11}$ with entries
\[
x_{ij}=\frac{a_{ij}}{\sigma_{ij}+\lambda\sigma_{ij}^{-1}};
\]
    \Ensure \[X=V_B\begin{bmatrix}X_{11} & 0\\
0 & 0\end{bmatrix}U_C^\tp.
\]
  \end{algorithmic}
\end{algorithm}

\subsection{Generalized two-sided product constrained approximation}

We now consider the problem \ref{it:le}:
\begin{equation}\label{eq:FXH}
\min_{FXG=H}{\norm{A-BXC}}.
\end{equation}
This may be viewed as a least squares counterpart to various \emph{two-sided linear matrix equations} such as
\[
\left\{
\begin{aligned}
BXC &= A,\\
FXG &= H,
\end{aligned}
\right. \qquad \text{or} \qquad BXC + FXG = H,
\]
that have been studied in \cite{Chu1,Mit,Wan}.

There is no loss of generality in assuming that $F$ has full row rank and $G$ has full column rank. Otherwise we may simply take the reduced QR factorizations $F = Q_F R_F$ where $Q_F^\tp Q_F = I$ and $G^\tp = Q_G R_G$ where  $Q_G^\tp Q_G = I$; observe that $FXG = H$ then becomes $R_F X R_G^\tp = Q_F^\tp H Q_G$, i.e., of the form $F'XG' = H'$ where $F' \coloneqq R_F$ has full row rank, $G' \coloneqq R_G^\tp$ has full column rank, and $H' \coloneqq Q_F^\tp H Q_G$.

For a closed-form solution, we will need to assume that  $B$ has full column rank and $C$ has full row rank, i.e., $\rank(B)=p$, $\rank(C)=q$.  Unlike the case of $F$ and $G$, there is a loss of generality in imposing these conditions on $B$ and $C$. However, the case of rank deficient $B$ and $C$ could be easily solved with our iterative algorithm in Section~\ref{sec:iter}.

We start with the simpler version
\begin{equation}\label{eq:linear}
\min_{FXG=H}\norm{A-X}.
\end{equation}
We claim that the solution to  \eqref{eq:linear} is given by $\widehat{X}=A+F^\tp(FF^\tp)^{-1}(H-FAG)(G^\tp G)^{-1}G^\tp$. First observe that
\[
F\widehat{X}G=F\bigl[A+F^\tp(FF^\tp)^{-1}(H-FAG)(G^\tp G)^{-1}G^\tp\bigr]G=FAG+H-FAG=H.
\]
Next recall that if $\lVert \, \cdot \, \rVert_2$ denotes the spectral norm (matrix $2$-norm), then $\norm{YZ}\le \norm{Y}_2\norm{Z}$ for any $Y\in\mathbb{R}^{m \times n}$, $Z = [z_1,\dots,z_p] \in\mathbb{R}^{n \times p}$ as
\begin{align*}
\norm{YZ}^2&=\sum_{i=1}^p\norm{Yz_i}_2^2\le \sum_{i=1}^p\norm{Y}_2^2\norm{z_i}_2^2=\norm{Y}_2^2\biggl[\sum_{i=1}^p\norm{z_i}_2^2\biggr]
=\norm{Y}_2^2\norm{Z}^2.
\end{align*}
Thus for any $X$ satisfying $FXG=H$, we have
\begin{align*}
\norm{A-\widehat{X}}&=\norm{F^\tp(FF^\tp)^{-1}(H-FAG)(G^\tp G)^{-1}G^\tp}\\
&=\norm{F^\tp(FF^\tp)^{-1}F(A-X)G(G^\tp G)^{-1}G^\tp }\\
&\le \norm{F^\tp(FF^\tp)^{-1}F}_2\norm{A-X}\norm{G(G^\tp G)^{-1}G^\tp }_2=\norm{A-X}
\end{align*}
since $F^\tp(FF^\tp)^{-1}F$ and $G(G^\tp G)^{-1}G^\tp$ are orthogonal projectors.

Now for the generalized problem  \eqref{eq:FXH}. Following our notations at the beginning of Section~\ref{sec:frame}, given that we have assumed $\rank(B)=p$ and $\rank(C)=q$, we have
\[
S_B = \Sigma_B, \quad S_C = \Sigma_C, \quad
\widetilde{X}=X_{11} ,\quad
\widetilde{A}=A_{11}.
\]
Let $F_B\coloneqq FV_B\Sigma_B^{-1}$ and $G_C \coloneqq\Sigma_C^{-1}U_C^\tp G$. The constraint $FXG=H$ is then equivalent to
\[
F_B(\Sigma_B\widetilde{X}\Sigma_C)G_C=H.
\]
By the solution to \eqref{eq:linear}, we get
\[
\widetilde{X}=\Sigma_B^{-1}\bigl[ \widetilde{A}+F_B^\tp(F_BF_B^\tp)^{-1}(H-F_B\widetilde{A}G_C)(G_C^\tp G_C)^{-1}G_C^\tp \bigr]\Sigma_C^{-1}.
\]
We summarize this solution in Algorithm~\ref{algo:gl}.

\begin{algorithm}[htb]
  \caption{Generalized two-sided product constrained approximation}\label{algo:gl}
  \begin{algorithmic}[1]
    \Require $A\in\mathbb{R}^{m \times n}$, $B\in\mathbb{R}^{m \times p}$, $C\in\mathbb{R}^{q \times n}$, $H\in\mathbb{R}^{k \times l}$, $F\in\mathbb{R}^{k \times p}$, $G\in\mathbb{R}^{q \times l}$;
    \State compute singular value decompositions $B=U_B\begin{bsmallmatrix}\Sigma_B\\ 0\end{bsmallmatrix}V_B^\tp $ and $C=U_C\begin{bmatrix}\Sigma_C& 0\end{bmatrix}V_C^\tp $;
    \State compute
\[
U_B^\tp AV_C = \widetilde{A};
\]
    \State compute $F_B=FV_B\Sigma_B^{-1}$ and $G_C=\Sigma_C^{-1}U_C^\tp G$;
    \State compute 
\[
\widetilde{X}=\Sigma_B^{-1}\bigl[\widetilde{A}+F_B^\tp(F_BF_B^\tp)^{-1}(H-F_B\widetilde{A}G_C)(G_C^\tp G_C)^{-1}G_C^\tp\bigr]\Sigma_C^{-1};
\]
    \Ensure $X=V_B\widetilde{X}U_C^\tp$.
  \end{algorithmic}
\end{algorithm}

\section{Closed-form solutions via GSVD}\label{sec:symm}

We now address \ref{it:symm}, starting with symmetry constraint:
\begin{equation}\label{eq:symm}
\min_{X=X^\tp }{\norm{A-BXC}}
\end{equation}
and deferring skew-symmetry constraint to later. Since $X$ is necessarily a square matrix, we require the number of columns in $B$ to equal the number of rows in $C$. So let $A\in\mathbb{R}^{m \times p}$, $B\in\mathbb{R}^{m \times n}$, $C\in\mathbb{R}^{n \times p}$. The special case where $C = I$ is called the symmetric Procrustes problem \cite{Brock,Lar} and was solved in \cite{Don,Hig2}. Exact versions of this problem, i.e., seeking symmetric solutions to $BXC  = A$, have been studied in \cite{Chu2,Kha}. The special case $B = C =I$ is a elementary and well-known: the projection of $A$ onto the spaces of symmetric and skew-symmetric matrices are given by the additive decomposition $A = (A+A^\tp)/2+ (A-A^\tp)/2$ into two orthogonal components.

To preserve the symmetry, the singular value decomposition is not useful as $V_B^\tp XU_C$ is generally not symmetric for a symmetric $X$. However, the generalized singular value decomposition  \cite{GSVD} is perfect for our purpose. We remind the reader of this result.
\begin{theorem}[Paige--Saunders]\label{thm:gsvd}
Let $B\in\mathbb{R}^{m \times n}$, $C\in\mathbb{R}^{n \times p}$ with
$k=\rank\bigl(\begin{bsmallmatrix}B\\C^\tp \end{bsmallmatrix}\bigr)$. Then there exist orthogonal matrices $U\in\mathbb{R}^{m \times m}$, $V\in\mathbb{R}^{p \times p}$, $W\in\mathbb{R}^{k \times k}$, and $Q\in\mathbb{R}^{n \times n}$ with
\begin{equation}\label{eq:88}
U^\tp BQ =\Sigma_B\begin{bmatrix}W^\tp R & 0\end{bmatrix}, \qquad V^\tp C^\tp Q =\Sigma_C\begin{bmatrix}W^\tp R & 0\end{bmatrix}
\end{equation}
where
\[
\Sigma_B=\begin{bmatrix}I_B & & \\ & S_B & \\ & & O_B  \end{bmatrix}, \qquad \Sigma_C =\begin{bmatrix}O_C & & \\ & S_C & \\ & & I_C  \end{bmatrix},
\]
and $R\in\mathbb{R}^{k \times k}$ is nonsingular with singular values equal to the nonzero singular values of $\begin{bsmallmatrix}B \\C^\tp \end{bsmallmatrix}$. Here $I_B\in\mathbb{R}^{r \times r}$ and $I_C\in\mathbb{R}^{k-r-s  \times k-r-s}$ are identity matrices, $O_B\in\mathbb{R}^{(m-r-s) \times (k-r-s)}$ and $O_C\in\mathbb{R}^{(p-k+r) \times r}$ are zero matrices with possibly no rows or columns, and $S_B=\diag(\beta_{r+1},\dots,\beta_{r+s})$, $S_C=\diag(\gamma_{r+1},\dots,\gamma_{r+s}) \in \mathbb{R}^{s \times s}$ with
\[
1>\beta_{r+1}\ge \dots\ge \beta_{r+s}>0, \qquad 0<\gamma_{r+1}\le \dots\le \gamma_{r+s}<1.
\]
\end{theorem}
Following Theorem~\ref{thm:gsvd}, let
\[
M=Q\begin{bmatrix}R^{-1}W & 0\\0 & I\end{bmatrix}.
\]
Then \eqref{eq:88} becomes
\[
U^\tp BM=\begin{bmatrix}\Sigma_B & 0\end{bmatrix}
, \qquad V^\tp C^\tp M=\begin{bmatrix}\Sigma_C & 0\end{bmatrix}.
\]
Let $\widetilde{A}=U^\tp AV$, $\widetilde{X}=M^{-1}XM^{-\tp}$. Then
\[
\lVert A-BXC\rVert = \bigl\lVert A-U\sb M^{-1}XM^{-\tp}\begin{bmatrix}\Sigma_C & 0\end{bmatrix}^\tp V^\tp \bigr\rVert
=\biggl\lVert \widetilde{A}-\begin{bmatrix}\Sigma_B & 0\end{bmatrix}\widetilde{X}\begin{bmatrix}\Sigma_C^\tp \\0\end{bmatrix} \biggr\rVert.
\]
Partition $\widetilde{A}$ and $\widetilde{X}$ to conform to the block structure of $\Sigma_B$ and $\Sigma_C^\tp $:
\[
\widetilde{A}=\begin{bmatrix}A_{11} & A_{12} & A_{13} \\ A_{21} & A_{22} & A_{23} \\ A_{31} & A_{32} & A_{33}\end{bmatrix},\qquad
\widetilde{X}=\begin{bmatrix}X_{11} & X_{12} & X_{13} & X_{14} \\ X_{21} & X_{22} & X_{23} & X_{24} \\ X_{31} & X_{32} & X_{33} & X_{34} \\ X_{41} & X_{42} & X_{43} & X_{44}\end{bmatrix}.
\]
We get
\[
\widetilde{A}-\begin{bmatrix}\Sigma_B & 0\end{bmatrix}\widetilde{X}\begin{bmatrix}\Sigma_C^\tp \\0\end{bmatrix}
=
\begin{bmatrix}
A_{11} & A_{12}-X_{12}S_C & A_{13}-X_{13}\\
A_{21} & A_{22}-S_BX_{22}S_C & A_{23}-S_BX_{23} \\
A_{31} & A_{32} & A_{33}
\end{bmatrix}
\]
and \eqref{eq:symm} reduces to minimizing
\[
\norm{A_{12}-X_{12}S_C}^2+\norm{A_{13}-X_{13}}^2+\norm{A_{23}-S_BX_{23}}^2+\norm{A_{22}-S_BX_{22}S_C}^2.
\]
Since $\widetilde{X}^\tp =\widetilde{X}$, the solution is easily seen to be
\begin{alignat*}{3}
X_{12}&=A_{12}S_C^{-1},\qquad & X_{13}&=A_{13},\qquad & X_{23}&=S_B^{-1}A_{23},\\
X_{21}&=S_C^{-\tp}A_{12}^\tp,\qquad & X_{31}&=A_{13}^\tp,\qquad& X_{32}&=A_{23}^\tp S_B^{-\tp};
\end{alignat*}
$X_{22}=(x_{ij}) \in \mathbb{R}^{s \times s}$ is given by defining $\sigma_{ij} \coloneqq \beta_{r+i}\gamma_{r+j}$ and setting
\[
x_{ij}=x_{ji}=
\begin{cases}
\dfrac{a_{ij}\sigma_{ij}+a_{ji}\sigma_{ji}}{\sigma_{ij}^2+\sigma_{ji}^2}, & \sigma_{ij}^2+\sigma_{ji}^2\neq 0,\\
0 &\text{otherwise},
\end{cases}
\]
for $i,j=1,\dots,s$. The other blocks $X_{11}, X_{33}, X_{44}, X_{14} = X_{41}^\tp, X_{24} = X_{42}^\tp, X_{34} = X_{43}^\tp$ are all set to be zero matrices. Note that although we have set the free parameters in $X_{22}$ to be zeros, they may be arbitrary as long as $X_{22}^\tp =X_{22}$. We summarize our solution  in Algorithm~\ref{algo:gs}.

\begin{algorithm}[!ht]
  \caption{Generalized symmetry constrained approximation}\label{algo:gs}
  \begin{algorithmic}[1]
    \Require $A\in\mathbb{R}^{m \times n}$, $B\in\mathbb{R}^{m \times p}$, $C\in\mathbb{R}^{q \times n}$, $k\ge 0$;
    \State compute generalized singular value decomposition
\[
U^\tp BQ=\Sigma_B\begin{bmatrix}W^\tp R & 0\end{bmatrix}, \qquad V^\tp C^\tp Q=\Sigma_C\begin{bmatrix}W^\tp R & 0\end{bmatrix}
\]
with $\Sigma_B=\diag(I_B, S_B, O_B)$, $\Sigma_C=\diag(O_C, S_C, I_C)$;
    \State compute
\[
M=Q\begin{bmatrix}R^{-1}W & 0\\0 & I\end{bmatrix};
\]
    \State compute
\[
U^\tp AV =\begin{bmatrix}A_{11} & A_{12} & A_{13} \\ A_{21} & A_{22} & A_{23} \\ A_{31} & A_{32} & A_{33}\end{bmatrix};
\]
    \State compute $X_{22} = (x_{ij})$ as
\[
x_{ij}=x_{ji}=
\begin{cases}
\dfrac{a_{ij}\sigma_{ij}+a_{ji}\sigma_{ji}}{\sigma_{ij}^2+\sigma_{ji}^2}, & \sigma_{ij}^2+\sigma_{ji}^2\neq 0,\\
0 &\text{otherwise};
\end{cases}
\]
    \Ensure \[\widehat{X}=M\begin{bmatrix}0 & A_{12}S_C^{-1} & A_{13} & 0 \\ S_C^{-\tp}A_{12}^\tp  & X_{22} & S_B^{-1}A_{23} & 0 \\ A_{13}^\tp  & A_{23}^\tp S_B^{-\tp} & 0 & 0 \\ 0 & 0 & 0 & 0\end{bmatrix}M^\tp.
\]
  \end{algorithmic}
\end{algorithm}

It is easy to adapt the solution above for skew-symmetric matrices. The only change being that for $\widetilde{X}^\tp =-\widetilde{X}$, we want
\begin{alignat*}{3}
X_{12}&=A_{12}S_C^{-1},\qquad & X_{13}&=A_{13},\qquad & X_{23}&=S_B^{-1}A_{23},\\
X_{21}&=-S_C^{-\tp}A_{12}^\tp,\qquad & X_{31}&=-A_{13}^\tp,\qquad& X_{32}&=-A_{23}^\tp S_B^{-\tp};
\end{alignat*}
$X_{22}=(x_{ij}) \in \mathbb{R}^{s \times s}$ is given by defining $\sigma_{ij} \coloneqq \beta_{r+i}\gamma_{r+j}$ and setting
\[
x_{ij}=-x_{ji}=\begin{cases}\dfrac{a_{ij}\sigma_{ij}-a_{ji}\sigma_{ji}}{\sigma_{ij}^2+\sigma_{ji}^2} & \sigma_{ij}^2+\sigma_{ji}^2\neq 0,\\
0&\text{otherwise},
\end{cases}
\]
for $i,j=1,\dots,s$.

\section{Iterative solution}\label{sec:iter}

The other problems \ref{it:struct} and \ref{it:pos} and the prescribed eigenvector problem take the form
\begin{equation}
\label{eq:pro}
\min_{X\in\mathcal{S}}\;\norm{A-BXC}
\end{equation}
where $\mathcal{S}$ is a closed convex set of matrices having the requisite property. Although we are unable to obtain a closed-form solution for these problems directly, we may solve them by alternating between two subproblems with closed-form solutions:
\begin{enumerate}[\upshape (a)]
\item\label{LS} for any given $A,B,C,Y$ and $\rho$, we have a closed-form solution for the norm constrained problem
\[
\min_{\norm{X-Y}\leq \rho }\norm{A - BXC};
\]

\item\label{Pr} for any $A$, when $B = I$ and $C = I$, we have a closed-form solution for the special case
\[
\min_{Y\in\mathcal{S}}\;\norm{A - Y}.
\]
\end{enumerate}
The problem \eqref{LS} is a minor variant of the problem \ref{it:norm} solved in Section~\ref{sec:norm}, with the $\rho$-ball centered at $Y$ instead of $0$. The problem \eqref{Pr} is a projection onto the set of interest $\mathcal{S}$, which we will solve in Section~\ref{sec:proj} for  \ref{it:struct}, \ref{it:pos}, and the prescribed eigenvector problem.
Sometimes we will have to project twice to different sets  $\mathcal{S}_1$ and $\mathcal{S}_2$; this happens when there are closed-form expressions for projections onto $\mathcal{S}_1$ and $\mathcal{S}_2$ but none for the set of interest $\mathcal{S} = \mathcal{S}_1 \cap \mathcal{S}_2$. An example is when $\mathcal{S}$ is the set of correlation matrices, with $\mathcal{S}_1$ the set of positive semidefinite matrices and $\mathcal{S}_2$ the set of matrices with ones on the diagonal \cite{Hig1}. We will discuss this variant in Section~\ref{sec:repeat}.

Alternating between \eqref{LS} and \eqref{Pr} tradeoffs between minimizing $\norm{A-BXC}$ and staying close to within distance $\rho$ of the projection $Y \in \mathcal{S}$. However if we simply alternate between these two subproblems, the iterates may end up simply oscillating between two fixed points. The well-known solution is to introduce a \emph{Dykstra correction} $Z$ \cite{Dyk} so that we have:
\begin{align}
Y_{k+1}&=\argmin\; \{ \norm{X_k- Z_k - Y} : Y \in \mathcal{S} \},\label{eq:sub1}\\
X_{k+1}&=\argmin\; \{\norm{A-BXC} : \norm{X-Z_k-Y_{k+1}} \leq \rho_k, \; X \in \mathbb{R}^{p \times q}\},\label{eq:sub2}\\
Z_{k+1}&= Z_k-X_{k+1} + Y_{k+1}.\label{eq:sub3}
\end{align}
In Theorems~\ref{thm:conv} and \ref{thm:linconv}, we prove the linear and global convergence of this iterative algorithm to a global minimizer of \eqref{eq:pro}.

We will see in Section~\ref{sec:proj} that \eqref{eq:sub1} is readily solvable for \ref{it:struct}, \ref{it:pos}, and the prescribed eigenvector problem.
If we set  $W=Z_k+Y_{k+1}$, $X'=X-W$, $A'=A-BWC$, then step \eqref{eq:sub2} becomes
\[
\min_{\norm{X'}\leq \rho_k}\;\norm{A'-BX'C},
\]
i.e., it is exactly the norm constrained problem \ref{it:norm} that we solved in Section~\ref{sec:norm}.  For the sequence of $\rho_k$, let $a_{ij}(W)$ denote the $(i,j)$th entry of $A'=A-BWC$. As we discussed in Section~\ref{sec:norm}, for any fixed $W$ there is a bijection between $\rho > 0$ and the largest root $\lambda > 0$ of the secular equation
\begin{equation}\label{eq:10a}
f(\lambda, W) \coloneqq \sum_{i=1}^s\sum_{j=1}^t \frac{a_{ij}(W)^2}{(\sigma_{ij}+\lambda\sigma_{ij}^{-1})^2}=\rho^2.
\end{equation}
We will show in Theorem~\ref{thm:linconv} that when $B$ has full column rank and $C$ has full row rank, choosing
\begin{equation}\label{eq:lambda}
\lambda =  \sigma_{\min}(B)\sigma_{\min}(C)\sigma_{\max}(B)\sigma_{\max}(C)
\end{equation}
gives us the fastest rate of convergence. In this case we have
\[
\rho(W) = \sqrt{f(\sigma_{\min}(B)\sigma_{\min}(C)\sigma_{\max}(B)\sigma_{\max}(C),W)}.
\]
Note that we set $\lambda$ to be a fixed constant but $\rho$ generally depends on $W$. We write $\rho(W)$ to emphasize this dependence; in \eqref{eq:sub2}, $\rho_k = \rho(Z_k + Y_{k+1})$. Nevertheless, like Algorithm~\ref{algo:gn}, $\rho_k$ will not make an appearance in our iterative algorithm, which only requires $\lambda$. Unlike Algorithm~\ref{algo:gn}, our iterative algorithm fixes a value of $\lambda$ at the beginning,  saving us the effort of solving a secular equation.

We summarize the above discussion in Algorithm~\ref{algo:iter}. Aside from lines~4 and 10, the rest of the algorithm is identical to  Algorithm~\ref{algo:gn} but sans the secular equation step. Yet another advantage is that the optimal $\lambda$ in \eqref{eq:lambda} is available to us ``for free'' since the algorithm requires computing the singular value decompositions of $B$ and $C$ to solve \eqref{eq:sub2}.

\begin{algorithm}
  \caption{Iterative algorithm for generalized nearness problems}\label{algo:iter}
  \begin{algorithmic}[1]
    \Require $A\in\mathbb{R}^{m \times n}$, $B\in\mathbb{R}^{m \times p}$, $C\in\mathbb{R}^{q \times n}$;
    \State precompute $B=U_B\begin{bsmallmatrix}\Sigma_B\\ 0\end{bsmallmatrix}V_B^\tp $ and $C=U_C\begin{bmatrix}\Sigma_C& 0\end{bmatrix}V_C^\tp $, $\sigma_{ij}=\sigma_i(B)\sigma_j(C)$;
    \State set $\lambda =  \sigma_{\min}(B)\sigma_{\min}(C)\sigma_{\max}(B)\sigma_{\max}(C)$;
    \State initialize $X_0$, $Y_0$, $Z_0$, $k=0$;
    \State compute $Y_{k+1}$ by projecting $X_k-Z_k$ to the desired set $\mathcal{S}$;
    \State compute $W=Y_{k+1}+Z_k$; 
    \State compute $A'=A-BWC$;
    \State compute $A_{11}$ from
\[
U_B^\tp A' V_C = \begin{bmatrix}A_{11} & A_{12} \\ A_{21} & A_{22} \end{bmatrix};
\]
    \State compute $X_{11} = (x_{ij})$ as
\[x_{ij}=\frac{a_{ij}}{\sigma_{ij}+\lambda\sigma_{ij}^{-1}};
\]
    \State compute \[X_{k+1}=V_B\begin{bmatrix}X_{11} & 0\\
0 & 0\end{bmatrix}U_C^\tp+W;\]
    \State compute $Z_{k+1}=Z_k-X_{k+1}+Y_{k+1}$;
    \State $k=k+1$ and go to line $4$;
  \end{algorithmic}
\end{algorithm}

\subsection{Convergence theorems}\label{sec:conv}

We will now show that the iterates generated by Algorithm~\ref{algo:iter} always converge to the global solution of \eqref{eq:pro} for any initialization and that the convergence rate is linear.
By \eqref{eq:5}, a solution $X_{k+1}$ of \eqref{eq:sub2} satisfies $B^\tp(BX_{k+1}C-A)C^\tp+\lambda(X_{k+1}-Z_k-Y_{k+1})=0$; plugging into \eqref{eq:sub3}, we get
\begin{equation}\label{eq:it1}
\lambda Z_{k+1}= B^\tp(BX_{k+1}C-A)C^\tp.
\end{equation}
A solution $Y_{k+1}$ of \eqref{eq:sub1} clearly satisfies
\[
0 \le \norm{X_k - Z_k - Y}^2 - \norm{X_k - Z_k - Y_{k+1}}^2 = 2 \tr\bigl((Y_{k+1}-X_k+Z_k)^\tp (Y-Y_{k+1})\bigr)  + \norm{Y-Y_{k+1}}^2
\]
for all $Y\in \mathcal{S}$. Thus $\tr\bigl((Y_{k+1}-X_k+Z_k)^\tp (Y-Y_{k+1})\bigr) \ge 0$ for all $Y\in \mathcal{S}$. Plugging into \eqref{eq:sub3}, we get
\begin{equation}\label{eq:it2}
\lambda\tr\bigl((X_k-X_{k+1}-Z_{k+1})^\tp(Y_{k+1}-Y)\bigr) \ge 0
\end{equation}
for all $Y\in \mathcal{S}$.

Let $X_*$ denote a global minimizer of $\norm{A-BXC}^2$ in $\mathcal{S}$. Note that the existence of $X_*$ is guaranteed since $\mathcal{S}$ is closed in all our considered choices and the function $X \mapsto \norm{A-BXC}^2$ has bounded sublevel sets. For any $Y\in \mathcal{S}$, since $X_*$ is a global minimizer we must have
\begin{align}
0 &\le \norm{A-BYC}^2-\norm{A-BX_*C}^2 \nonumber\\
&=2\tr\bigl((B^\tp(BX_*C-A)C^\tp)^\tp (Y-X_*)\bigr)+\norm{B(Y-X_*)C}^2 \label{eq:globmin}
\end{align}
and thus
\[
\tr\bigl((B^\tp(BX_*C-A)C^\tp)^\tp (Y-X_*)\bigr)\ge 0.
\]
Let $Z_*$ be defined by
\begin{equation}\label{eq:Z1}
\lambda Z_* \coloneqq B^\tp(BX_*C-A)C^\tp,
\end{equation}
where we introduce the parameter $\lambda$ for consistency with \eqref{eq:it1}.
Then the last inequality becomes
\begin{equation}\label{eq:Z2}
\lambda\tr\bigl(Z_*^\tp (Y-X_*)\bigr) \ge 0.
\end{equation}
Conversely, if $X_*$ is such that \eqref{eq:Z2} holds for all $Y \in \mathcal{S}$, then $X_*$ must be a global minimizer by virtue of \eqref{eq:globmin}.

\begin{lemma}\label{lem:sd}
The iterate $(X_k,Z_k)$ of Algorithm~\ref{algo:iter} satisfies
\begin{multline}\label{eq:sd}
\norm{X_k-X_*}^2+\norm{Z_k-Z_*}^2-\norm{X_{k+1}-X_*}^2-\norm{Z_{k+1}-Z_*}^2\\
\ge \norm{X_k-X_{k+1}}^2+\norm{Z_k-Z_{k+1}}^2+\frac{2}{\lambda}\sigma_{\min}(B)^2\sigma_{\min}(C)^2\norm{X_k-X_{k+1}}^2\\+\frac{2}{\lambda}\sigma_{\min}(B)^2\sigma_{\min}(C)^2\norm{X_{k+1}-X_*}^2.
\end{multline}
\end{lemma}
\begin{proof}
It is easy to see \begin{align}
\label{eq:sc}
\begin{split}
\sigma_{\max}(B)^2\sigma_{\max}(C)^2\norm{X_1-X_2}^2&\ge \tr\bigl[\bigl(B^\tp B(X_1-X_2)CC^\tp\bigr)^\tp\bigl(X_1-X_2\bigr)\bigr],\\
\sigma_{\min}(B)^2\sigma_{\min}(C)^2\norm{X_1-X_2}^2&\le \tr\bigl[\bigl(B^\tp B(X_1-X_2)CC^\tp\bigr)^\tp\bigl(X_1-X_2\bigr)\bigr].
\end{split}
\end{align}
Combining \eqref{eq:Z1}, \eqref{eq:sc}, and \eqref{eq:it1}, we have
\begin{equation}\label{eq:3.1}
\lambda\tr\bigl((X_{k+1}-X_*)^\tp(Z_{k+1}-Z_*)\bigr)\ge \sigma_{\min}(B)^2\sigma_{\min}(C)^2\norm{X_{k+1}-X_*}^2.
\end{equation}
Combining \eqref{eq:Z2} and \eqref{eq:it2}, we have
\begin{equation}
\lambda\tr\bigl((Y_{k+1}-X_*)(X_k-X_{k+1}-Z_{k+1}+Z_*)\bigr)\ge 0\label{eq:3.2}.
\end{equation}
Adding \eqref{eq:3.1} and \eqref{eq:3.2}, and applying \eqref{eq:sub3}, we have
\begin{multline}
\label{eq:26}
\lambda\tr\bigl[\bigl(Z_k-Z_{k+1}\bigr)^\tp\bigl(Z_{k+1}-Z_*-(X_{k}-X_{k+1})\bigr)\bigr]\\+\lambda\tr\bigl((X_{k+1}-X_*)^\tp(X_{k}-X_{k+1})\bigr)
\ge \sigma_{\min}(B)^2\sigma_{\min}(C)^2\norm{X_{k+1}-X_*}^2.
\end{multline}
Dividing \eqref{eq:26} by $\lambda$ and using the identity $\norm{a-c}^2-\norm{b-c}^2=2\tr\bigl((a-c)^\tp(a-b)\bigr)-\norm{a-b}^2$,
\begin{multline}\label{eq:4}
\norm{X_k-X_*}^2+\norm{Z_k-Z_*}^2-\norm{X_{k+1}-X_*}^2-\norm{Z_{k+1}-Z_*}^2
\ge \norm{X_k-X_{k+1}}^2\\+\norm{Z_k-Z_{k+1}}^2+2\tr\bigl((Z_k-Z_{k+1})^\tp (X_k-X_{k+1})\bigr)+\frac{2}{\lambda}\sigma_{\min}(B)^2\sigma_{\min}(C)^2\norm{X_{k+1}-X_*}^2.
\end{multline}
From \eqref{eq:sc} and \eqref{eq:it1}, we have 
\[
\lambda\tr\bigl((Z_k-Z_{k+1})^\tp(X_k-X_{k+1})\bigr)\ge \sigma_{\min}(B)^2\sigma_{\min}(C)^2\norm{X_k-X_{k+1}}^2.
\]
Plug this into \eqref{eq:4} and we obtain the desired result.
\end{proof}

\begin{theorem}[Global convergence]\label{thm:conv}
For any $\lambda>0$, the iterate $(X_k,Z_k)$ of Algorithm~\ref{algo:iter} converges to some $(X_*, Z_*)$ satisfying \eqref{eq:Z1} and \eqref{eq:Z2}.
\end{theorem}
\begin{proof}
By \eqref{eq:sd}, the sequence $(\norm{X_k-X_*}^2+\norm{Z_k-Z_*}^2)_{k =0}^\infty$ is monotone decreasing and thus convergent. This also implies $(X_k)_{k =0}^\infty$ and $(Z_k)_{k =0}^\infty$ are bounded, and therefore have convergent subsequences $(X_{k_j})_{j =0}^\infty$ with limit $\widetilde{X}$ and $(Z_{k_j})_{j =0}^\infty$ with limit $\widetilde{Z}$.
Taking limits in \eqref{eq:it1} and \eqref{eq:it2} over these subsequences, we obtain
\[
\lambda \widetilde{Z} =B^\tp(B\widetilde{X}C-A)C^\tp, \qquad \lambda\tr\bigl(\widetilde{Z}^\tp(Y-\widetilde{X})\bigr) \ge 0
\]
for all $Y\in \mathcal{S}$.
Thus $(\widetilde{X}, \widetilde{Z})$ satisfies \eqref{eq:Z1} and \eqref{eq:Z2}, implying that $\widetilde{X}$ is a global minimizer. Let $X_*=\widetilde{X}$ and $Z_*=\widetilde{Z}$. Since  $(\norm{X_k-X_*}^2+\norm{Z_k-Z_*}^2)_{k =0}^\infty$ is convergent, we must have $X_k\rightarrow X_*$ and $Z_k\rightarrow Z_*$.
\end{proof}

\begin{theorem}[Linear convergence]\label{thm:linconv}
Suppose $B$ has full column rank and $C$ has full row rank. Then \[\norm{X_k-X_*}^2+\norm{Z_k-Z_*}^2\ge (1+\delta)(\norm{X_{k+1}-X_*}^2+\norm{Z_{k+1}-Z_*}^2)\]
with 
\[
\delta=\biggl[\frac{\lambda}{2\sigma_{\min}(B)^2\sigma_{\min}(C)^2}+\frac{\sigma_{\max}(B)^2\sigma_{\max}(C)^2}{2\lambda}\biggr]^{-1}.
\]
In particular, choosing
\[
\lambda=\sigma_{\min}(B)\sigma_{\min}(C)\sigma_{\max}(B)\sigma_{\max}(C)
\]
maximizes the convergence rate with
\[
\delta =\frac{\sigma_{\min}(B)\sigma_{\min}(C)}{\sigma_{\max}( B)\sigma_{\max}(C)} = \frac{1}{\kappa_2(B)\kappa_2(C)}.
\]
\end{theorem}
\begin{proof}
If $B$ has full column rank and $C$ has full row rank, then 
\[
\sigma_{\max}(B)^2\sigma_{\max}(C)^2\ge \sigma_{\min}(B)^2\sigma_{\min}(C)^2>0.
\]
Since for any $X_1,X_2\in\mathbb{R}^{m \times n}$,
\[
\norm{B^\tp B(X_1-X_2)CC^\tp}^2\le \sigma_{\max}(B)^2\sigma_{\max}(C)^2\tr\bigl[\bigl(B^\tp B(X_1-X_2)CC^\tp\bigr)^\tp (X_1-X_2)\bigr],
\]
we have
\begin{equation}\label{eq:3.1a}
\lambda\tr\bigl((X_{k+1}-X_*)^\tp(Z_{k+1}-Z_*)\bigr)\ge \dfrac{1}{\sigma_{\max}(B)^2\sigma_{\max}(C)^2}\norm{\lambda Z_{k+1}-\lambda Z_*}^2.
\end{equation}
Taking a convex combination of \eqref{eq:3.1} and \eqref{eq:3.1a}, we get
\begin{multline}\label{eq:convex}
\lambda\tr\bigl((X_{k+1}-X_*)^\tp(Z_{k+1}-Z_*)\bigr)\ge t\sigma_{\min}(B)^2\sigma_{\min}(C)^2\norm{X_{k+1}-X_*}^2 \\+(1-t)\dfrac{1}{\sigma_{\max}(B)^2\sigma_{\max}(C)^2}\norm{\lambda Z_{k+1}-\lambda Z_*}^2
\end{multline}
for any $t\in[0,1]$. When we derived \eqref{eq:sd}, the term $2\sigma_{\min}(B)^2\sigma_{\min}(C)^2\norm{X_{k+1}-X_*}^2/\lambda$ came from \eqref{eq:3.1}. If we use \eqref{eq:convex} in place of \eqref{eq:3.1} in our derivation, we obtain
\begin{multline*}
\norm{X_k-X_*}^2+\norm{Z_k-Z_*}^2-\norm{X_{k+1}-X_*}^2-\norm{Z_{k+1}-Z_*}^2\\
\ge \norm{X_k-X_{k+1}}^2+\norm{Z_k-Z_{k+1}}^2+\frac{2}{\lambda}\sigma_{\min}(B)^2\sigma_{\min}(C)^2\norm{X_{k}-X_{k+1}}^2\\
+t\frac{2}{\lambda}\sigma_{\min}(B)^2\sigma_{\min}(C)^2\norm{X_{k+1}-X_*}^2 \\
+(1-t)\dfrac{2}{\lambda \sigma_{\max}(B)^2\sigma_{\max}(C)^2}\norm{\lambda Z_{k+1}-\lambda Z_*}^2.
\end{multline*}
Now drop the nonnegative term $2\sigma_{\min}(B)^2\sigma_{\min}(C)^2\norm{X_{k}-X_{k+1}}^2/\lambda$ and combine the last two terms to get $\delta(\norm{X_{k+1}-X_*}^2+\norm{Z_{k+1}-Z_*})$ with
\[
t = \frac{\lambda^2}{1+\sigma_{\min}(B)^2\sigma_{\min}(C)^2\sigma_{\max}(B)^2\sigma_{\max}(C)^2}. \qedhere
\]
\end{proof}

\begin{corollary}\label{cor:rate}
The iterate $(X_k,Z_k)$ of Algorithm~\ref{algo:iter} satisfies
\[
\norm{X_k-X_{k+1}}^2+\norm{Z_{k}-Z_{k+1}}^2\le \frac{\norm{X_0-X_*}^2+\norm{Z_0-Z_*}^2}{k+1}.
\]
\end{corollary}
\begin{proof}
By subtracting two successive inequalities of the form \eqref{eq:it2}, we get
\[
\tr\bigl((Z_{k+1}+X_{k+1}-X_{k}-Z_{k+2}-X_{k+2}+X_{k+1})^\tp(Y_{k+2}-Y_{k+1})\bigr)\ge 0.
\]
Adding this to $2\lambda\tr\bigl((Z_{k+1}-Z_{k+2})^\tp(X_{k+1}-X_{k+2})\bigr)\ge 0$ gives
\begin{multline*}
\tr\bigl((X_{k+1}-X_{k+2})^\tp(X_{k}-2X_{k+1}+X_{k+2})\bigr)\\+
\tr\bigl((Z_{k+1}+X_{k+1}-X_{k}-Z_{k+2}-X_{k+2}+X_{k+1})^\tp(X_{k+1}-X_{k+2})\bigr)\ge 0,
\end{multline*}
It follows from \eqref{eq:sub3} that
\[
Z_{k}+Y_{k+1}-X_{k}-Z_{k+1}-Y_{k+2}+X_{k+1}=Z_{k+1}+X_{k+1}-X_{k}-Z_{k+2}-X_{k+2}+X_{k+1}.
\]
So the last inequality becomes
\begin{multline}\label{eq:5.7}
\tr\bigl((X_{k+1}-X_{k+2})^\tp(X_{k}-2X_{k+1}+X_{k+2})\bigr)\\+\tr\bigl((Z_{k}+Y_{k+1}-X_{k}-Z_{k+1}-Y_{k+2}+X_{k+1})^\tp(X_{k+1}-Y_{k+1}-X_{k+2}+Y_{k+2})\bigr)\ge 0.
\end{multline}
Adding
\[
\norm{X_{k}-2X_{k+1}+X_{k+2}}^2+\tr\bigl((X_{k}-Y_{k+1}-X_{k+1}+Y_{k+2})^\tp(X_{k+1}-Y_{k+1}-X_{k+2}+Y_{k+2})\bigr)
\]
to both sides of \eqref{eq:5.7}, we get
\begin{multline*}
\tr\bigl((X_k-X_{k+1})^\tp(X_{k}-2X_{k+1}+X_{k+2})\bigr)+\tr\bigl((Z_k-Z_{k+1})^\tp(Z_{k}-2Z_{k+1}+Z_{k+2})\bigr)\\
\ge \norm{X_{k}-2X_{k+1}+X_{k+2}}^2+\tr\bigl((X_{k}-Y_{k+1}-X_{k+1}+Y_{k+2})^\tp(X_{k+1}-Y_{k+1}-X_{k+2}+Y_{k+2})\bigr);
\end{multline*}
and with this inequality, we see that
\begin{align*}
\norm{X_k&-X_{k+1}}^2+\norm{Z_k-Z_{k+1}}^2-\norm{X_{k+1}-X_{k+2}}^2-\norm{Z_{k+1}-Z_{k+2}}^2\\
&=\begin{multlined}[t]
2\tr\bigl((X_k-X_{k+1})^\tp(X_{k}-2X_{k+1}+X_{k+2})\bigr)\\
+2\tr\bigl((Z_k-Z_{k+1})^\tp(Z_{k}-2Z_{k+1}+Z_{k+2})\bigr)\\
\qquad\qquad-\norm{X_k-2X_{k+1}+X_{k+2}}^2-\norm{Z_{k}-2Z_{k+1}+Z_{k+2}}^2
\end{multlined}\\
&\ge \begin{multlined}[t]
\norm{X_{k}-2X_{k+1}+X_{k+2}}^2\\
+2\tr\bigl((X_{k}-Y_{k+1}-X_{k+1}+Y_{k+2})^\tp(X_{k+1}-Y_{k+1}-X_{k+2}+Y_{k+2})\bigr)\\
-\norm{X_k-2X_{k+1}+X_{k+2}}^2-\norm{Z_{k}-2Z_{k+1}+Z_{k+2}}^2
\end{multlined}\\
&\ge \begin{multlined}[t]
2\norm{X_{k}-2X_{k+1}+X_{k+2}}^2\\
+2\tr\bigl((X_k-2X_{k+1}+X_{k+2}+Z_{k}-2Z_{k+1}+Z_{k+2})^\tp(Z_{k}-2Z_{k+1}+Z_{k+2})\bigr)\\
-\norm{X_k-2X_{k+1}+X_{k+2}}^2-\norm{Z_{k}-2Z_{k+1}+Z_{k+2}}^2
\end{multlined}\\
&=\norm{X_k-2X_{k+1}+X_{k+2}+Z_{k}-2Z_{k+1}+Z_{k+2}}^2\ge 0.
\end{align*}
It follows that the sequence $(\norm{X_k-X_{k+1}}^2+\norm{Z_{k}-Z_{k+1}}^2)_{k=0}^\infty$ is monotone decreasing. Hence
\begin{align*}
(k+1)(\norm{X_k-X_{k+1}}^2+\norm{Z_{k}-Z_{k+1}}^2) &\le \sum_{k=0}^\infty(\norm{X_k-X_{k+1}}^2+\norm{Z_{k}-Z_{k+1}}^2)\\
&\le \norm{X_0-X_*}^2+\norm{Z_0-Z_*}^2
\end{align*}
where the inequality follows from \eqref{eq:sd}.
\end{proof}

\subsection{Repeated projections}\label{sec:repeat}

There are occasions when we do not have a single closed-form solution for a projection onto the desired set $\mathcal{S} = \mathcal{S}_1 \cap \mathcal{S}_2$ but we do have closed-form solutions for a projections onto $\mathcal{S}_1$ and $\mathcal{S}_2$. Of course, one may then obtain an iterative method for projection onto $\mathcal{S}$ simply by alternating between projections onto $\mathcal{S}_1$ and $\mathcal{S}_2$ \cite{Dyk}. Nevertheless, in situations like this, standard wisdom from the design of iterative algorithms informs us that it would be better to intertwine these inexpensive projections onto $\mathcal{S}_1$ and $\mathcal{S}_2$ with other steps of Algorithm~\ref{algo:iter} --- instead of projecting onto $\mathcal{S}$, an expensive endeavor requiring a separate iterative procedure, in every iteration of Algorithm~\ref{algo:iter}. With this in mind, we obtain the following ``two projections variant'' of Algorithm~\ref{algo:iter}:
\begin{equation}\label{eq:iter2}
\begin{aligned}
X_{k+1}&=\argmin\; \{ \norm{A-BXC}^2+\lambda\norm{X-W_k+Z_k}^2 : X \in \mathbb{R}^{p \times q} \},\\
Y_{k+1}&=\argmin\; \{\norm{ W_k - Z'_k - Y} : Y \in \mathcal{D} \},\\
W_{k+1}&=\argmin\; \{\norm{X_{k+1}+Y_{k+1}+Z_k+Z'_k)/2 - W} : W \in \mathcal{S} \},\\
Z_{k+1}&=Z_k+X_{k+1}-W_{k+1},\\
Z'_{k+1}&=Z'_k+Y_{k+1}-W_{k+1}.
\end{aligned}
\end{equation}
It is straightforward to extend this to include three (or more) projections when we have $\mathcal{S} = \mathcal{S}_1 \cap \mathcal{S}_2 \cap \mathcal{S}_3$.

The convergence results in Section~\ref{sec:conv} may also be adapted to \eqref{eq:iter2}. To account for the fact that we now have two Dykstra corrections $Z_k$ and $Z'_k$, we replace \eqref{eq:it1} and \eqref{eq:it2} by
\begin{align*}
\lambda (W_k-W_{k+1}-Z_{k+1}) &= B^\tp(BX_{k+1}C-A)C^\tp,\\
\lambda\tr\bigl((W_k-W_{k+1}-Z'_{k+1})^\tp(Y_{k+1}-Y)\bigr) &\ge 0,\\
\lambda\tr\bigl((Z_{k+1}+Z'_{k+1})^\tp(W_{k+1}-W)\bigr) &\ge 0
\end{align*}
for all $Y\in \mathcal{S}_1$ and $ W\in \mathcal{S}_2$. It is straightforward to check that with this modification the proofs of Lemma~\ref{lem:sd}, Theorems~\ref{thm:conv} and \ref{thm:linconv}, Corollary~\ref{cor:rate} carry through for the algorithm in \eqref{eq:iter2}.

\section{Computing projections}\label{sec:proj}

We rely on Algorithm~\ref{algo:iter} for the generalized nearness problems \ref{it:struct}, \ref{it:pos}, and the prescribed eigenvector problem. Since the algorithm alternates between projection and norm constrained least squares, it remains to discuss how we may compute a projection, i.e.,
\[
\min_{X\in\mathcal{S}}\;\norm{A - X}
\]
for the relevant sets $\mathcal{S}$. We first reminder readers that that the projection of $A \in \mathbb{R}^{n \times n}$ to the subspace of symmetric matrices $\mathbb{S}^n \coloneqq  \{X\in\mathbb{R}^{n \times n} : X = X^\tp\}$ is given by $X = (A + A^\tp)/2$, a fact that we will use liberally below.

For \ref{it:struct}, $\mathcal{S}$ is the subspace of either Toeplitz, Hankel, or circulant matrices \cite{Pan}:
\begin{align*}
\operatorname{Toep}_n(\mathbb{R}) &\coloneqq \{X\in\mathbb{R}^{n \times n} :
x_{i,i+k} = x_{j, j+k}, \; -n+1 \le k \le n-1, \; 1 \le k+i, k+j, i,j \le n \},\\
\operatorname{Hank}_n(\mathbb{R}) &\coloneqq \{X\in\mathbb{R}^{n \times n} : x_{i,k-i} = x_{j, k-j}, \; 2 \le k \le 2n, \; 1 \le k-i, k-j, i,j \le n \},\\
\operatorname{Circ}_n(\mathbb{R}) &\coloneqq \{X\in\mathbb{R}^{n \times n} :
x_{ij} = x_{kl},\; i-j \equiv k-l \bmod n\}.
\end{align*}
The projections of $A\in\mathbb{R}^{n \times n}$ onto $\operatorname{Toep}_n(\mathbb{R})$, $\operatorname{Hank}_n(\mathbb{R})$, $\operatorname{Circ}_n(\mathbb{R})$ are then given respectively by
\[
x_{ij}=\begin{cases}
\displaystyle\frac{1}{n-\abs{i-j}}\sum_{k-l=i-j} a_{kl} & \text{(Toeplitz)}, \\[6ex]
\displaystyle\frac{1}{n-\abs{i+j-n-1}}\sum_{k+l=i+j} a_{kl}  & \text{(Hankel)}, \\[6ex]
\displaystyle\frac{1}{n}\smash[b]{\sum_{k-l \equiv i-j \bmod{n}}} a_{kl}  & \text{(circulant)},
\end{cases}
\]
for $i,j = 1,\dots,n$.

For the prescribed eigenvector problem, given any nonzero $v \in \mathbb{R}^n$, we write
\[
\mathbb{S}^n_v \coloneqq \{X\in\mathbb{S}^{n} : Xv=\lambda v \;\text{for some}\; \lambda \in \mathbb{R}\}.
\]
For any given $A \in\mathbb{R}^{n \times n}$ and nonzero $v \in \mathbb{R}^n$, we may assume $\norm{v}_2=1$, and if not, just normalize. Let $V\in\mathbb{R}^{n \times n}$ be an orthogonal matrix whose first column is $v$. Let $e = [1,0,\dots,0]^\tp \in \mathbb{R}^n$. Then as $Xv=\lambda v$ iff $V^\tp XVe=\lambda e$, the projection problem reduces to
\[
\min_{X\in\mathbb{S}^n_v}{\norm{A-X}} =\min_{\widetilde{X}\in\mathbb{S}^n_{e}}{\norm{V^\tp AV-\widetilde{X}}}.
\]
With this observation, a projection of $A \in \mathbb{R}^{n \times n}$ to $X \in\mathbb{S}^n_v$ may be computed by first extending $v$ to an orthogonal matrix $V\in\mathbb{R}^{n \times n}$ (e.g., by using QR decomposition); partitioning 
\[
V^\tp AV =\begin{bmatrix}
\alpha_{11} & a_{12}^\tp \\
a_{21} & A_{22}
\end{bmatrix}
\]
with $\alpha_{11}\in\mathbb{R}$, $a_{12}, a_{21} \in \mathbb{R}^{n-1}$, and $A_{22}\in\mathbb{R}^{(n-1)\times (n-1)}$; and finally computing
\[
X = V\begin{bmatrix}a_{11} & 0\\
0 & (A_{22}+A_{22}^\tp )/2\end{bmatrix}V^\tp.
\]

It remains to address the projections for \ref{it:pos}. We introduce more standard notations: nonnegative and positive semidefinite $X \in \mathbb{R}^{n \times n}$ are denoted
\[
X \ge 0, \qquad X \succeq 0
\]
respectively, i.e., the former means that $x_{ij} \ge 0$ for all $i,j=1,\dots,n$ whereas the latter means that the quadratic form $v^\tp X v \ge 0$ for all $v \in \mathbb{R}^n$. The convex sets of nonnegative matrices and of symmetric positive semidefinite matrices are denoted 
\[
\mathbb{R}^{n \times n}_\p \coloneqq \{X \in \mathbb{R}^{n \times n} : X \ge 0 \},\qquad
\mathbb{S}^n_\p \coloneqq \{X \in \mathbb{S}^n :  X \succeq 0 \}
\]
respectively. It is well-known that the projections of a matrix $A\in\mathbb{R}^{m \times n}$ onto these sets are given by simply zeroing out negative entries or negative eigenvalues. For $\mathbb{R}^{m \times n}_\p$, it is an exercise: the projection $X \in \mathbb{R}^{n \times n}_\p$ is given by
\[
X=\max(A,0)
\]
where $\max(\,\cdot\,, 0)$ is applied coordinatewise to a matrix. For $\mathbb{S}^n_\p$, it is slightly more involved  \cite{Hig4}: the projection $X\in \mathbb{S}^n_\p$ is given by taking the symmetric eigenvalue decomposition  $(A+A^\tp)/2=V\Lambda V^\tp $ and setting
\[
X =V\max(\Lambda,0) V^\tp.
\]
The same projection $X$ may also be computed with polar decomposition  \cite{Hig4,polar}.

The convex sets of stochastic, doubly stochastic, and correlation matrices  \cite{Horn} are
\begin{align*}
\operatorname{Stoc}_n(\mathbb{R}) &\coloneqq \Bigl\{X\in\mathbb{R}^{n \times n}_\p:
\sum_{j=1}^n x_{ij} = 1, \;  i=1,\dots,n\Bigr\},\\
\operatorname{DStoc}_n(\mathbb{R}) &\coloneqq \Bigl\{X\in\mathbb{R}^{n \times n}_\p:
\sum_{j=1}^n x_{ij} = \sum_{i=1}^n x_{ij} = 1, \; i,j=1,\dots,n\Bigr\}, \\
\operatorname{Corr}_n(\mathbb{R}) &\coloneqq \{X\in\mathbb{S}^n_\p: x_{ii} = 1, \; i=1,\dots,n\}.
\end{align*}
By our discussion in Section~\ref{sec:repeat}, we only need to address the question of projections onto
\begin{align*}
\mathcal{C} &\coloneqq \Bigl\{X\in\mathbb{R}^{n \times n}: \sum_{j=1}^n x_{ij} = 1, \;  i=1,\dots,n\Bigr\},\\
\mathcal{R} &\coloneqq \Bigl\{X\in\mathbb{R}^{n \times n} : \sum_{i=1}^n x_{ij} = 1, \; j=1,\dots,n\Bigr\}, \\
\mathcal{D} &\coloneqq \{X\in\mathbb{S}^n : x_{ii} = 1, \; i=1,\dots,n\}
\end{align*}
since
\[
\operatorname{Stoc}_n(\mathbb{R})  = \mathcal{C} \cap  \mathbb{R}^{n \times n}_\p, \qquad
\operatorname{DStoc}_n(\mathbb{R})  = \mathcal{C} \cap \mathcal{R} \cap  \mathbb{R}^{n \times n}_\p, \qquad
\operatorname{Corr}_n(\mathbb{R})  = \mathcal{D} \cap  \mathbb{S}^n_\p.
\]
The projections of $A\in\mathbb{R}^{n \times n}$ onto $\mathcal{C}$, $\mathcal{R}$,  $\mathcal{D}$ are easily seen to be given respectively by
\[
x_{ij}=  a_{ij}-\frac{1}{n}\biggl[\sum_{k=1}^n a_{ik}-1\biggr] ,\quad x_{ij} = a_{ij}-\frac{1}{n}\biggl[\sum_{k=1}^n a_{kj}-1\biggr], \quad x_{ij} =
\begin{cases}
1, & i=j,\\
a_{ij} & i \ne j,
\end{cases}
\]
where $i,j = 1,\dots,n$.

\section{Numerical Experiments}

We compare Algorithm~\ref{algo:iter}, which is based on numerical linear algebra (\textsc{nla}), with general algorithms based on convex optimization (\textsc{cvx}). We use problem \ref{it:pos} for illustration as these nonnegativity constraints are more complex and allow us to test our two-projection variant \eqref{eq:iter2}. We will minimize $\norm{A - BXC}$ with $X$ constrained to nonnegative, stochastic, positive semidefinite, and correlation matrices. The first two require just a single projection and thus Algorithm~\ref{algo:iter} suffices; the last two require two projections and thus call for \eqref{eq:iter2}. We compare our results against those obtained with convex optimization methods in \textsc{cvx} \cite{cvx}, using its \textsc{ecos} solver for the nonnegative and stochastic cases and its \textsc{scs} solver for the positive semidefinite and correlation cases.

We generate random matrices $B, C, X \in\mathbb{R}^{n \times n}$, with $X$ satisfying the constraint at hand, and set $A=BXC$. This represents the most common scenario --- we minimize  $\norm{A - BXC}$ when we really want to solve $A = BXC$ in the presence of errors. Furthermore, as $X$ is known and (almost surely) unique, we may use the forward error $\norm{\widehat{X}-X}/\norm{X}$  as a metric to make comparisons.

\subsection{Speed}

For each dimension $n=2^d$, we repeat our runs ten times and compare average time taken to reach a prespecified forward error. The default precisions of \textsc{cvx} are \texttt{reltol = abstol = feastol = 1e-8} for \textsc{ecos} and \verb|eps_rel = eps_abs = 1e-4| for \textsc{scs}; we scale these parameters by the dimension of the matrix $n$. We record the final forward error of \textsc{cvx} and run Algorithm~\ref{algo:iter} until it achieves the same forward error. From Figure~\ref{fig:plot2}, Algorithm~\ref{algo:iter} outperforms \textsc{cvx} significantly in speed for large $n$. Indeed, the range of dimensions is limited by \textsc{cvx}, which fails to converge for large $n$. To get a rough idea, for $n=2^7$, \textsc{cvx} took about half an hour when Algorithm~\ref{algo:iter} took less than a second. For $n = 2^8$ and beyond, \textsc{cvx} did not converge within 24 hours.

\begin{figure}[htb]% trim={<left> <lower> <right> <upper>}
    \includegraphics[width=0.49\textwidth,trim={3ex 1ex 8ex 5ex},clip]{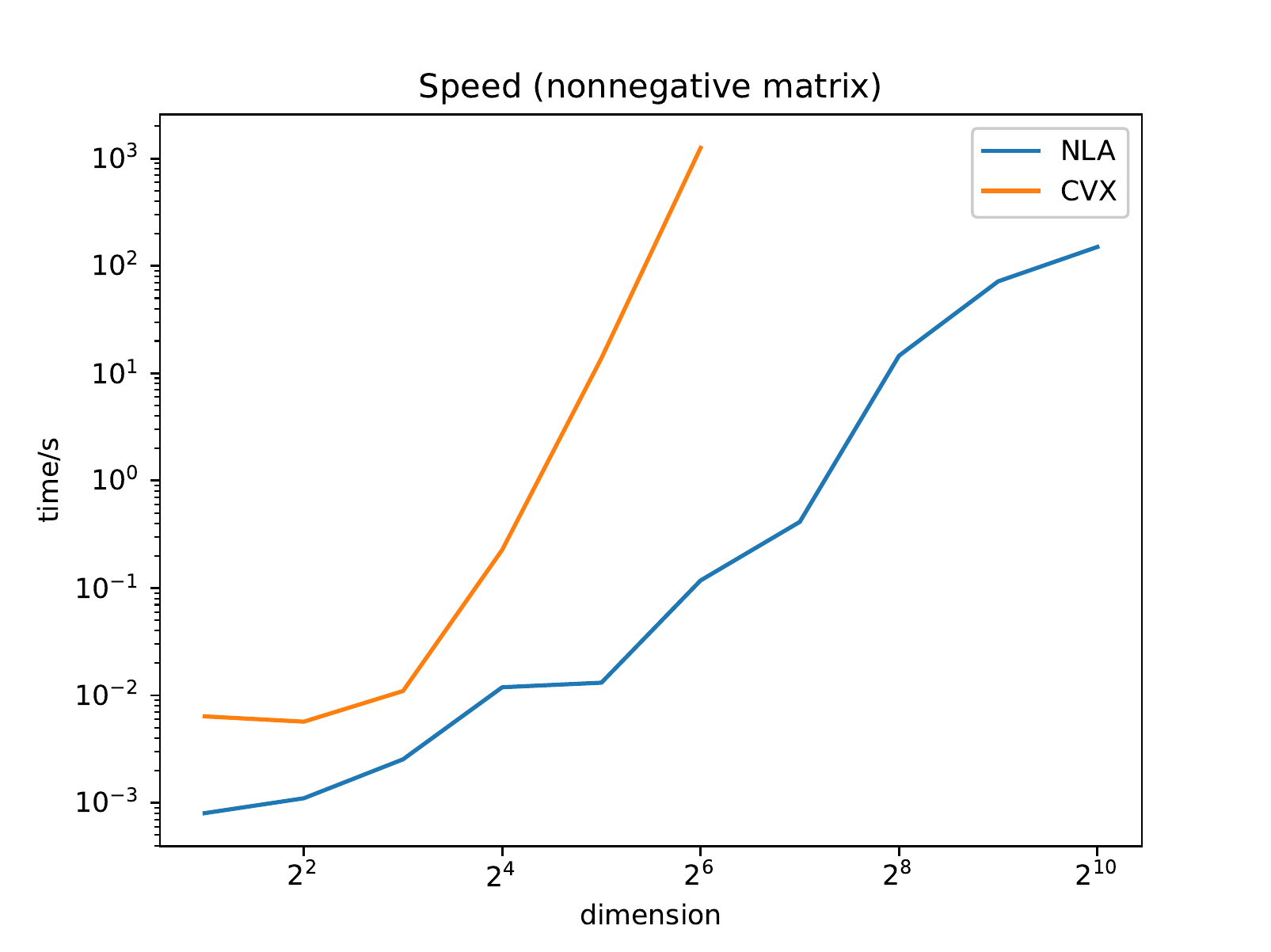}
    \includegraphics[width=0.49\textwidth,trim={2ex 1ex 9ex 5ex},clip]{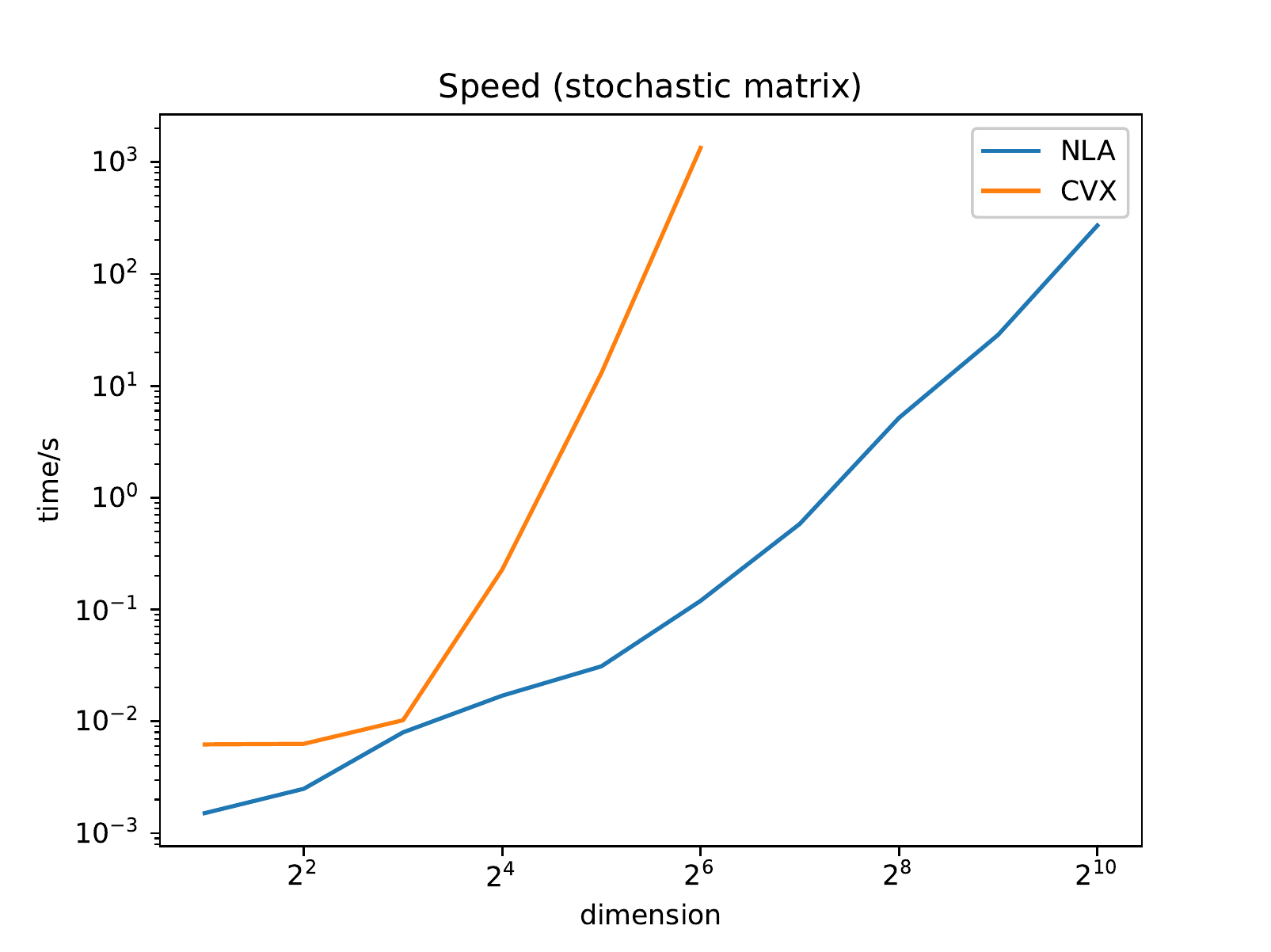}
    \includegraphics[width=0.49\textwidth,trim={3ex 2ex 8ex 3ex},clip]{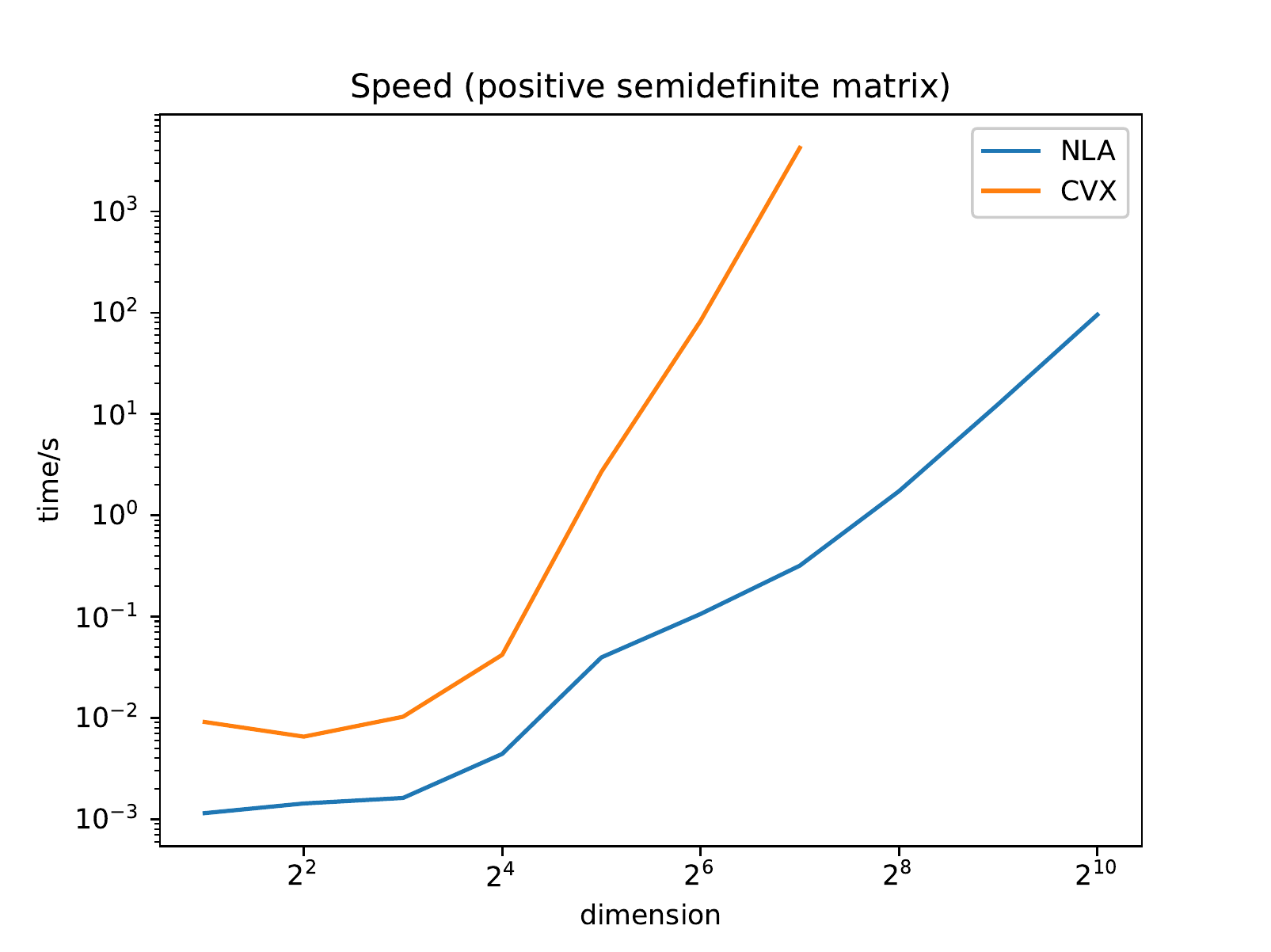}
    \includegraphics[width=0.49\textwidth,trim={2ex 2ex 9ex 3ex},clip]{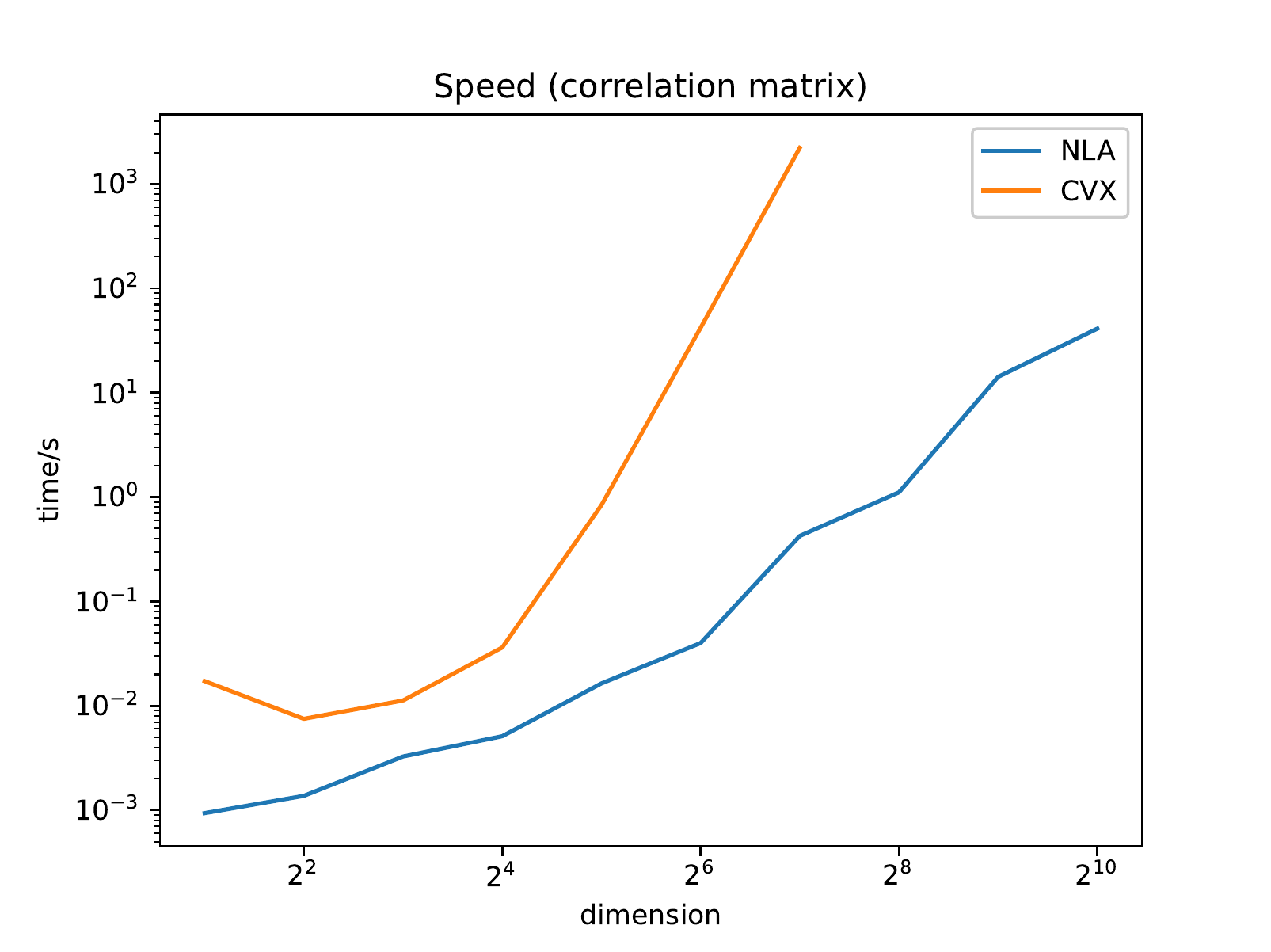}
\caption{Speed of Algorithm~\ref{algo:iter} versus \textsc{cvx}.}
\label{fig:plot2}
\end{figure}

These results are within expectation: in convex optimization, these problems are transformed to convex quadratic programs or semidefinite programs; both require at least $O(n^4)$ operations per iteration, which is prohibitive for large $n$. In our Algorithm~\ref{algo:iter} and its two-projection variant \eqref{eq:iter2}, the dominating cost is the projection onto $\mathcal{S}$ --- this is essentially free for nonnegative and stochastic matrices, requiring only $O(n^2)$ operations per iteration; for correlation and positive semidefinite matrices, projections require around $O(n^3)$ operations per iteration.

\subsection{Accuracy}

Here we compare the minimum possible forward error each method can achieve. We set $n=32$ as \textsc{cvx} may fail to converge in a reasonable amount of time for larger values of $n$. We set \textsc{cvx} to its maximum allowed precisions, \verb|reltol = abstol = feastol = 1e-16| for \textsc{ecos} and \verb|eps_rel = eps_abs = 1e-16| for \textsc{scs}, and record its \emph{final} forward error. Then we run Algorithm~\ref{algo:iter} for 5,000 iterations and record its forward error \emph{at each iteration}. Note that it is not meaningful to make iterationwise comparisons here as the two algorithms are entirely different. From Figure~\ref{fig:plot1}, we see that Algorithm~\ref{algo:iter} reaches beyond the maximum possible accuracy of \textsc{cvx}  in every case. The linear convergence rate in Theorem~\ref{thm:linconv} is also evident from these plots.

\begin{figure}[htb]% trim={<left> <lower> <right> <upper>}
    \includegraphics[width=0.49\textwidth,trim={1ex 2ex 8ex 5ex},clip]{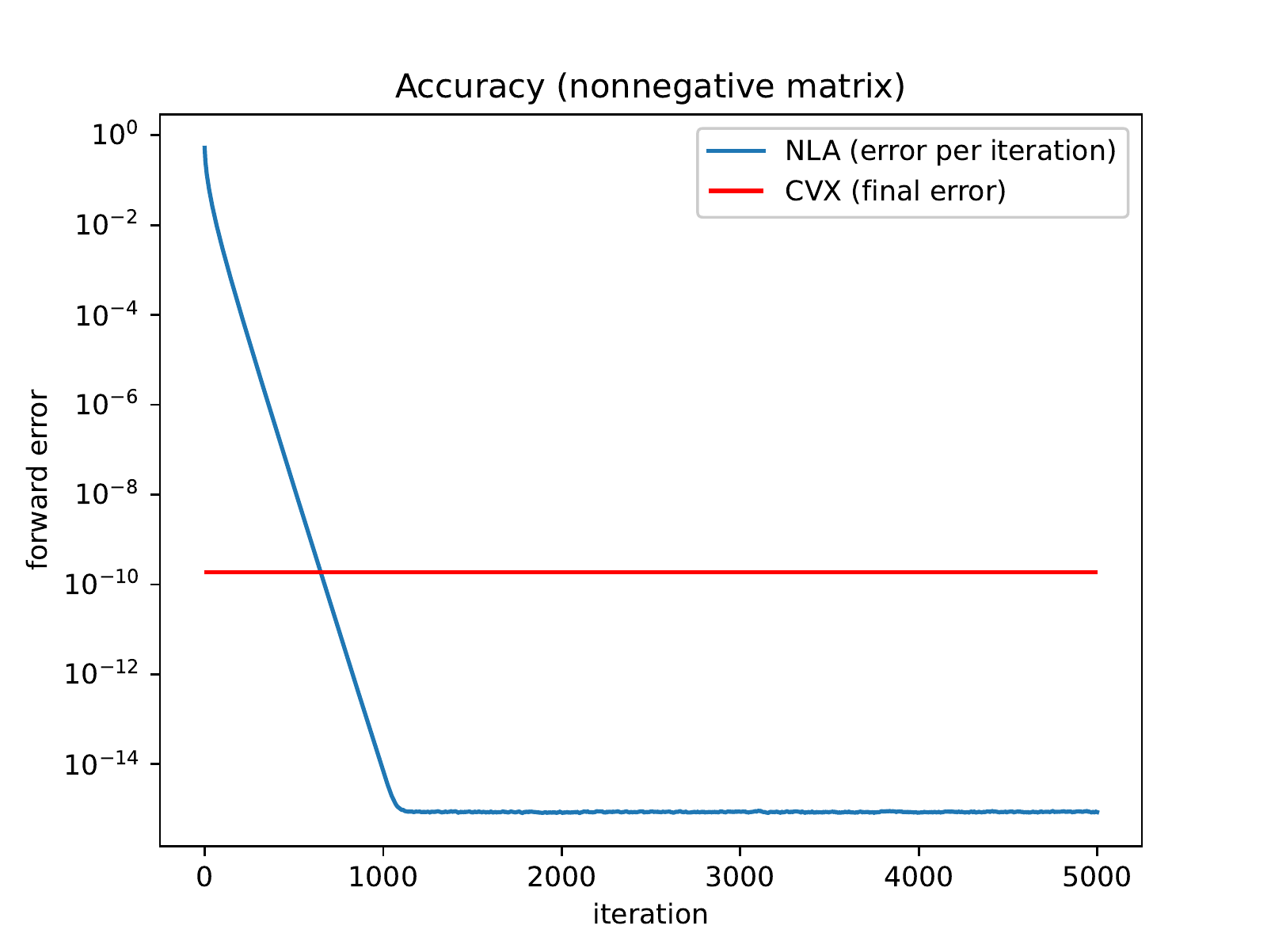}
    \includegraphics[width=0.49\textwidth,trim={2ex 2ex 9ex 5ex},clip]{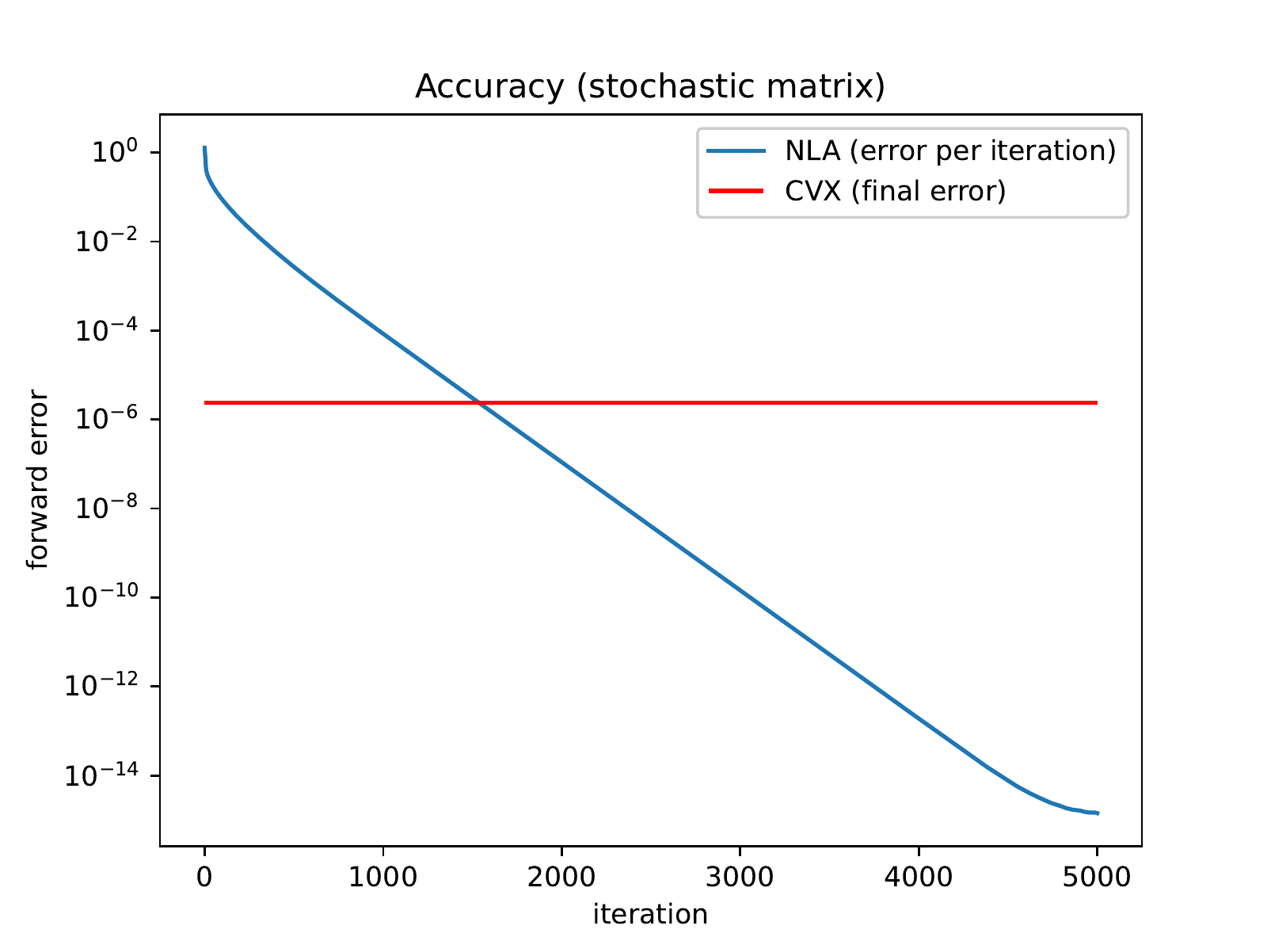}
    \includegraphics[width=0.49\textwidth,trim={1ex 2ex 8ex 3ex},clip]{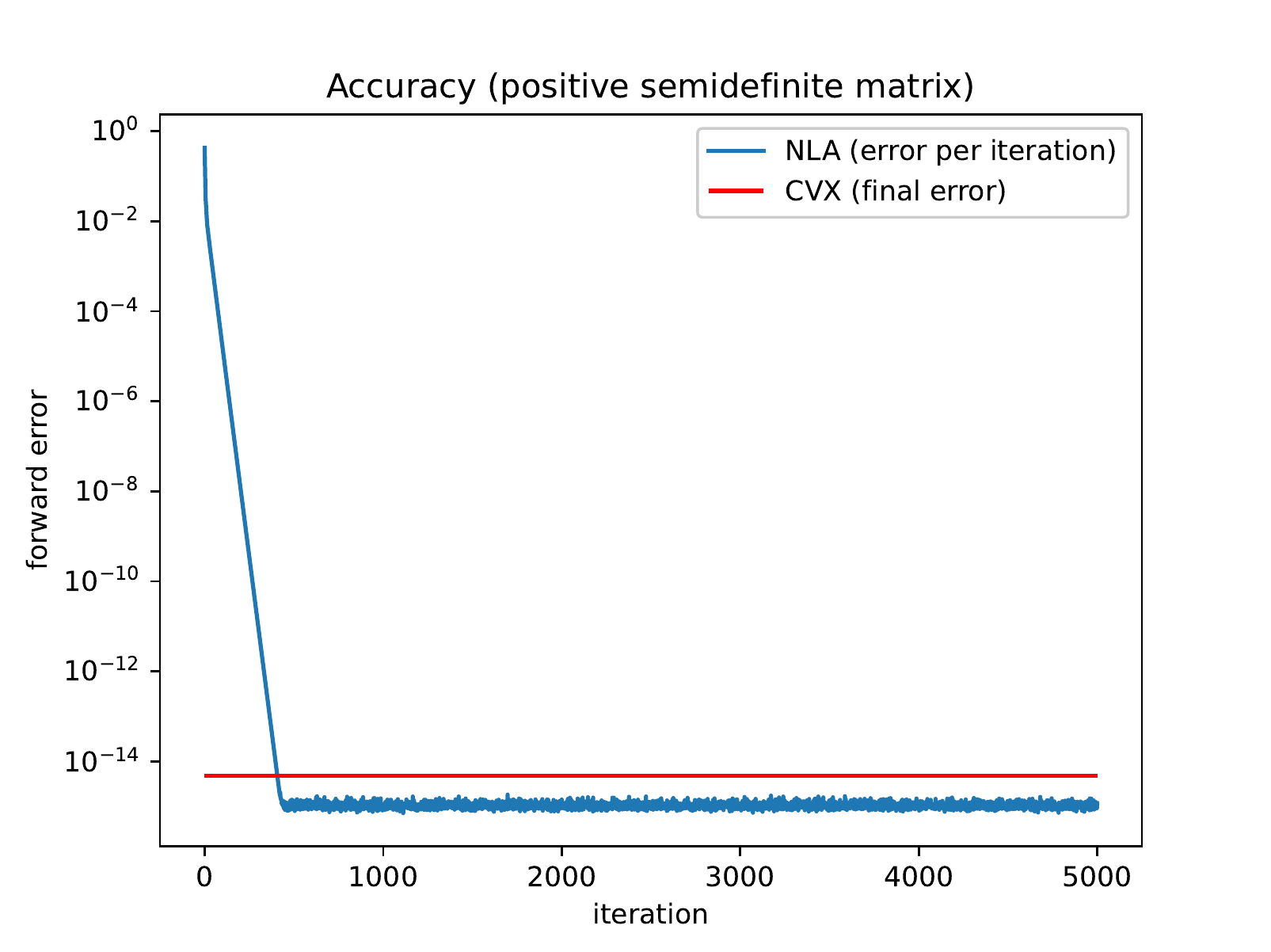}
    \includegraphics[width=0.49\textwidth,trim={2ex 2ex 9ex 3ex},clip]{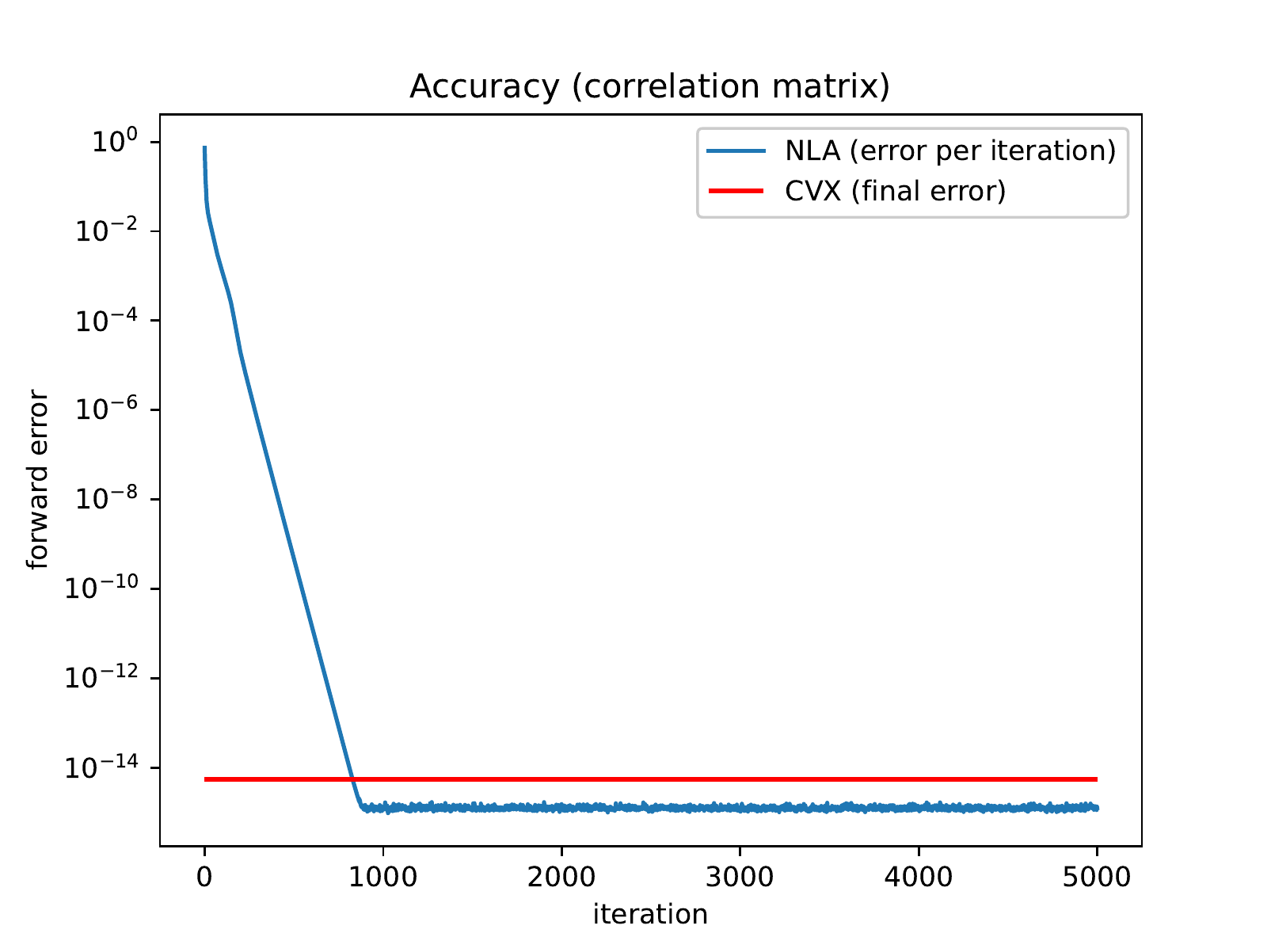}
\caption{Accuracy of Algorithm~\ref{algo:iter} versus \textsc{cvx}. Here the red lines indicate the level of maximum possible accuracy of \textsc{cvx}.}
\label{fig:plot1}
\end{figure}

\section{Conclusion}

Likely because of the increasing awareness of convex optimization as a potent tool in many areas, there has been a tendency to apply general purpose convex optimization methods to any convex problem. However convex problems like those considered in this article often have more structures than mere convexity. We show that approaching such problems through numerical linear algebra in the spirit of \cite{Gol1,Gol2,Hig4,Hig2,Hig3,Hig1,Kell,Rao,Wil2} can sometimes lead to better results, and has the advantage of working for the occasional nonconvex problem like \eqref{eq:Fri}.

\subsection*{Acknowledgment} The first author thanks Bartolomeo Stellato for helpful discussions.

\bibliographystyle{abbrv}

\end{document}